\documentclass[12pt]{amsart}

\usepackage{amsmath, amsthm, amssymb, amsfonts, mathrsfs, amscd, tikz, bbm}
\usepackage{enumerate, paralist, hyperref, color, ulem}
\usepackage{dsfont}
\usepackage[all,cmtip]{xy}
\usepackage{ulem}

\def\AA{{\mathbb A}}

\def\CC{{\mathbb C}}

\def\NN{{\mathbb N}}

\def\PP{{\mathbb P}}

\def\RR{{\mathbb R}}

\def\ZZ{{\mathbb Z}}

\def\0{{\mathbf 0}}
\def\1{{\mathbf 1}}

\def\SL{\mathrm{SL}}

\def\Rat{\mathrm{Rat}}

\def\PGL{\mathrm{PGL}}

\def\GL{\mathrm{GL}}

\def\ord{\mathrm{ord}}

\def\Res{\mathrm{Res}}

\def\min{\mathrm{min}}
\def\uf{\mathrm{uf}}
\def\uc{\mathrm{uc}}

\def\Ratfd2{\mathrm{Rat}_{d,2}^{\uf}}
\def\Mfd2{\mathcal{M}_{d,2}^{\uf}}
\def\Rfd2{\mathcal{R}_{d,2}^{\uf}}
\def\Rcd2{\mathcal{R}_{d,2}^{\uc}}

\def\Md{\mathcal{M}_d}

\def\Ratd{\mathrm{Rat_d}}

\DeclareMathOperator{\ordRes}{ordRes}
\DeclareMathOperator{\MinResLoc}{MinResLoc}
\DeclareMathOperator{\Bary}{Bary}
\newcommand{\pberk}{\mathbf{P}^1_K}
\newcommand{\hberk}{\mathbf{H}^1_K}
\newcommand{\aberk}{\mathbf{A}^1_K}
\newcommand{\zetaG}{\zeta_G}

\newcommand{\vv}{\vec{v}}

\theoremstyle{plain}

\newtheorem{thm}{Theorem}
\newtheorem*{thm*}{Theorem}

\newtheorem{cor}[thm]{Corollary}
\newtheorem*{cor*}{Corollary}
\newtheorem{prop}[thm]{Proposition}
\newtheorem{lem}[thm]{Lemma}
\newtheorem*{prop*}{Proposition}
\newtheorem*{remark*}{Remark}

\theoremstyle{definition}

\newtheorem{notation}{Notation}

\title{Iteration and the Minimal resultant}

\author{Kenneth Jacobs}
\author{Phillip Williams}



\begin{document}

\begin{abstract}
Let $K$ be an algebraically closed field that is complete with respect to a non-Archimedean absolute value, and  let $\varphi\in K(z)$ have degree $d\geq 2$. We characterize maps for which the minimal resultant of an iterate $\varphi^n$ is given by a simple formula in terms of $d$, $n$, and the minimal resultant of $\varphi$. We show that such maps are precisely those with reduction outside of an indeterminacy locus $I(d)$ and which also have semi-stable reduction for every iterate $\varphi^n$. We give two equivalent ways of describing such maps, one measure theoretic and the other in terms of the moduli space $\mathcal{M}_d$ of degree $d$ rational maps.

As an application, we are able to give an explicit formula for the minimal value of the diagonal Arakelov-Green's function of a map satisfying the conditions of the main theorem. We illustrate our results with some explicit calculations in the case of the Latt\`es maps.
\end{abstract}

\maketitle

\section*{Introduction}

Let $K$ be a complete, algebraically closed non-Archimedean valued field. We will denote the ring of integers by $\mathcal{O}$, with maximal ideal $\mathfrak{m}$. The residue field will be written $k=\mathcal{O}/\mathfrak{m}$. If $\textrm{char}(k) = 0$ let $q_v=e$ be the base of the natural logarithm; otherwise let $q_v$ be the residue characteristic. We normalize the absolute value on $K$ so that $-\log_v |x| = \ord_{\mathfrak{m}}(x)$, where $\log_v = \log_{q_v}$. 

Let $\varphi\in K(z)$ have degree $d\geq 2$. A homogeneous lift of $\varphi$ is a pair of coprime homogeneous polynomials $\Phi=  [F,G]$, say \begin{align*}
F(X,Y) & = a_d X^d+ ... + a_0 Y^d\\
G(X,Y) & = b_d X^d + ... + b_0 Y^d\ ,
\end{align*} with the property that $\varphi(z) = \frac{F(z,1)}{G(z,1)}$. A lift $[F,G]$ is said to be normalized if $\max(|a_i|, |b_i|) = 1$.  We will often identify the map $\varphi$ with a point in $\PP^{2d+1}$ via the identification $\varphi \mapsto [a_d : ... :a_0:b_d:...:b_0] =:[a: b]$, which is clearly independent of the choice of lift.

The resultant $\Res(F,G)$ of a lift of $\varphi$ is a homogeneous polynomial in the coefficients of $F, G$ of degree $2d$, which we can also regard as a function of $\PP^{2d+1}$ using the identification above. We will write $R_\varphi$ for the ord value of the resultant of a normalized lift of $\varphi$. The {\it minimal resultant} is a conjugacy invariant of $\varphi$ given 
\[
R_{[\varphi]} := \min_{\gamma\in \PGL_2(K)}\ R_{\varphi^\gamma}\ (\geq 0)\ ,
\] where $\varphi^\gamma = \gamma^{-1}\circ\varphi \circ \gamma$ is the usual conjugacy action. We say that $\varphi$ has good reduction if $R_{\varphi} = 0$, and that $\varphi$ has potential good reduction if $R_{[\varphi]} = 0$.

The minimal resultant has appeared in the work of several other authors. Silverman \cite{silverman:ads} gives an overview of the minimal resultant and asked questions about the existence of a global minimal model and about Northcott-type properties related to the minmal resultant. These questions were subsequently explored in work Rumely \cite{Ru1} and of Stout and Townsley \cite{StoutTownsley}. Szpiro, Tepper, and the second author \cite{STW} have explored the connections between the minimality of the resultant and semistability in the sense of GIT, as has Rumely \cite{Rumely}. The first author has explored how the conjugates attaining the minimal resultant vary for higher iterates of the map \cite{KJThesis}.

In this paper, we are interested in understanding how the minimal resultant of an iterate $\varphi^n$ map relates to the minimal resultant of the original map. The resultant form itself behaves nicely under iteration: it is a power (that is a simple formula in terms of $n$ and $d$) of the resultant of the original map (see Lemma \ref{resultant} below). Two things, however, get in the way of the {\it minimal} resultant from behaving so nicely. The first is the normalization that may have to take place in order to insure that not all coefficients vanish under reduction: even if the coefficients for a lift of $\varphi$ are normalized, the coefficients obtained by iteration need not be. The second is the potential change of coordinates that takes place to give the minimal valuation for the resultant, which need not be the same for every iterate.

We will draw on two tools for resolving these issues. The first is a notion of indeterminacy introduced by DeMarco in \cite{DeMarco1, DeMarco2}; the indeterminacy locus $I(d)\subseteq \PP^{2d+1}$ is the locus where the rational map $\Gamma_n :\PP^{2d+1} \dashrightarrow \PP^{2d^n+1}$ induced by iterating $\varphi$ is undefined. 

The second tool is geometric invariant theory, and particularly the connections between semistability, the minimality of the resultant, and the indeterminacy locus. Connections between semistability and the resultant were first explored by Szpiro, Tepper and the second author in \cite{STW}, and later by Rumely \cite{Rumely}, while DeMarco explored the connections between semistability and $I(d)$ in \cite{DeMarco1, DeMarco2}.

These tools will be applied in particular to the reduction of $\varphi$: given a normalized lift $[F,G]$ of $\varphi$, corresponding to a point $[a: b] \in \PP^{2d+1}$, let $[\tilde{a}: \tilde{b}]\in \PP^1(k)$ define the coordinates of a rational map $\varphi_v$ on $\PP^1(k)$; we emphasize that $\varphi_v$ may not be a morphism, as $[\tilde{a}: \tilde{b}]$ may give rise to polynomials that share a common factor.

Our main result is

\begin{thm*}

The minimal resultant iteration formula 
\begin{equation}\label{eq:minresitform} R_{[\varphi^n]} = \frac{d^n(d^n-1)}{d(d-1)} \cdot R_{[\varphi]}
\end{equation} holds if and only if in any coordinate system where $\varphi$ has semistable reduction, we have that $\varphi_v \not\in I(d)$ and $\varphi^n$ has semistable reduction for all $n$. 

\end{thm*}

In Section~\ref{examples}, we give examples of maps which satisfy (\ref{eq:minresitform}), as well as maps which fail to satisfy (\ref{eq:minresitform}). It is interesting to consider, then, what sorts of maps satisfy the equivalent conditions of this theorem. To explore this question, we employ both analytic and algebro-geometric ideas. 

Working over the Berkovich projective line, we are able to give two geometric conditions equivalent to (\ref{eq:minresitform}); the first condition is a stability property for Rumely's minimal resultant loci (\cite{Ru1, Rumely}), while the second condition asserts that (\ref{eq:minresitform}) holds if and only if no point in the barycenter of the equilibrium measure $\mu_\varphi$ corresponds to a conjugate with reduction in $I(d)$. The minimal resultant locus, the equilibrium measure, and the barycenter will be defined in Section~\ref{Berk}. Our proof of these conditions relies on the construction of two `residue' measures attached to $\varphi$; these measures first appeared (separately) in \cite{DeMarco1} and \cite{DF}. 

As an application of these equivalent conditions, we are able to compute the minimal value of the diagonal Arakelov-Green's function $g_\varphi(x,x)$ (defined in Section~\ref{mainresult}) for maps $\varphi$ satisfying the hypotheses of the theorem; in particular, we obtain
\begin{cor*}

If $\varphi$ satisfies either of the equivalent conditions in the preceeding Theorem, then $$\min_{x\in \pberk} g_\varphi(x,x) = \frac{1}{d(d-1)}R_{[\varphi]}\ .$$

In particular, in this case we find $\min_{x\in \pberk} g_\varphi(x,x) >0$ if and only if $\varphi$ does not have good reduction. 
\end{cor*} The last part of the corollary is already known to follow from a more general result of Baker \cite{Baker}, where the equivalence $\min_{x\in \pberk} g_\varphi(x,x)>0$ and bad reduction is established unconditionally. In Section~\ref{examples} below, we carry out these computations explicitly for flexible Latt\`es maps. Following a suggestion of Matt Baker, we compare this to the minimal value of the two-variable Arakelov-Green's function on an elliptic curve using results in Baker-Petsche \cite{BakerPetsche}.

We are also able to re-cast the semi-stability assumption in the main Theorem in terms of open subsets of parameter space, and we pose the question of whether this subset is Zariski dense; see Section~\ref{sect:AG}. At present the authors do not know whether this holds even for small examples. 

The outline for this paper is as follows: In Section \ref{indeterminacy} we introduce the necessary background regarding parameter space and reduction of rational maps. In Section~\ref{sect:lemmas} we establish preliminary lemmas concerning the resultant, semistability, and the indeterminacy locus $I(d)$. We prove the main result in Section~\ref{mainresult}. Following this, we recall some background on the Berkovich projective line, and establish the first equivalent condition to our main Theorem. Section~\ref{sect:AG} contains another equivalent condition to the main Theorem, this time using algebro-geometric machinery. In Section~\ref{sect:potentialtheory} we prove the corollary stated above pertaining to the minimal value of the Arakelov-Green's function. Finally, in Section~\ref{examples} we give concrete examples of maps where the equivalent conditions hold, and other maps where the conditions fail to hold.

\subsection*{Acknowledgements} The authors would like to thank Laura DeMarco and Pete Clark for helpful correspondence, and also Matt Baker for suggesting the connections between our work and the Green's functions on elliptic curves.

\section{Notation and Background} \label{indeterminacy}

\subsection{Iteration on Parameter Space}
Over any base, morphisms of degree $d$ on $\PP^1$ are parameterized by the coefficients of two homogeneous polynomials of degree $d$ without common roots. This last condition is equivalent to the non-vanishing of the resultant of the two polynomials, and so the space of rational maps of degree $d$ is the complement of the resultant hypersurface, an open subscheme of a projective space: $\Ratd \subset \PP^{2d+1}$. Points in $\PP^{2d+1}$ that are not in $\Ratd$ correspond to pairs of homogeneous polynomials with common roots; canceling these common roots yields a ``degenerate" map $\widetilde{\varphi}$ of lower degree.

Iteration of a rational map defines a morphism $\Gamma_n: \Ratd \rightarrow \mathrm{Rat}_{d^n}$. This map extends to a rational map on the projective spaces:  $$\Gamma_n: \PP^{2d+1} \dashrightarrow \PP^{2d^n+1}.$$ In \cite{DeMarco1}, DeMarco showed that, for every $n$, this map is defined outside of a set $I(d)$ of co-dimension $d+1$, and described precisely what this locus looks like. Though working over $\CC$, DeMarco's gives a completely algebraic characterization of the indeterminacy locus (\cite{DeMarco1}, Lemma 6) that works over base $\ZZ$. Her characterization of $I(d)$ as a set (\cite{DeMarco1}, Lemma 7) then works over any infinite field.

\begin{prop}
The set on which $\Gamma_n: \PP^{2d+1} \dashrightarrow \PP^{2d^n+1}$ is undefined consists, for every $n$, of the maps such that $\widetilde{\varphi}$ is constant and one of the factors that cancels is the constant value of $\widetilde{\varphi}$. 
\end{prop}
\begin{proof}
See \cite{DeMarco1}, Lemma 7.
\end{proof}

Crucially, $I(d)$ as a set doesn't depend on $n$. Throughout this paper, we will primarily be concerned with whether or not a rational map defined over the residue field lies in $I(d)$; as such, we will most often view $I(d)\subseteq\PP^1(k)$.

\subsection{Reduction and the Resultant Divisor}

Let $\varphi: \PP^1(K) \rightarrow \PP^1(K)$ be a morphism of degree $d$. As above, $\varphi$ can be represented by a point $[a,b] = [a_d, ..., a_0, b_d, ..., b_0] \in \Rat_d(K) \subseteq \PP^{2d+1}(K)$ in projective space; we let $$F(X,Y)=a_d X^d + ... +a_0Y^d\ , \ G(X,Y) = b_d X^d + ... + b_0Y^d$$ be coprime, homogeneous polynomials of degree $d$ that represent $\varphi$. We say that the representation $F,G$ is normalized  if each coefficient has absolute value at most one, and at least one coefficient has absolute value 1. Any representative can be made into a normalized representative if we divide through by the coefficient with the largest absolute value; on the other hand, normalized representatives are not unique: scaling by any unit will preserve normalization.

\begin{notation}Given a normalized representative $F,G$ of $\varphi$, we define the \textit{reduction of $\varphi$} to be the rational map of $\PP^1(k)$ given $$\varphi_v:=[\tilde{F}, \tilde{G}]\ ,$$ where $\tilde{F}, \tilde{G}$ are the polynomials over $k$ obtained by reducing the coefficients of $F, G$. On the parameter space $\PP^{2d+1}(K)$, this corresponds to reducing coordinates modulo $\mathfrak{m}$; if $\varphi$ corresponds to the point $[a,b]\in \PP^{2d+1}(K)$, the point corresponding to the reduction map is denoted $[\tilde{a}, \tilde{b}]\in \PP^{2d+1}(k)$.
\end{notation}

\begin{notation}The reduction is said to be \textit{degenerate} if the polynomials $\tilde{F}, \tilde{G}$ have a common factor. In this case, we write $\tilde{A} =$gcd$(\tilde{F}, \tilde{G})$.  Let $\tilde{F} = \tilde{A} \cdot \tilde{F}_0$ and $\tilde{G} = \tilde{A}\cdot \tilde{G}_0$. The \textit{residue map} $\widetilde{\varphi}$ of $\varphi$ is the morphism of $\PP^1(k)$ given by $$\widetilde{\varphi}:= [\tilde{F}_0, \tilde{G}_0]\ .$$ If the polynomials $\tilde{F}, \tilde{G}$ do not have a common factor, the residue map is defined to be the morphism $[\tilde{F}, \tilde{G}]$ of $\PP^1(k)$; in this case, $\varphi$ has good reduction.
\end{notation}

\begin{notation} \label{resultantdivisor} Given a rational map $\varphi\in \Rat_d(K)$, let $R_{\varphi}$ denote the ord-value of the resultant of a normalized lift of $\varphi$. Likewise, let $R_{[\varphi]}$ denote the minimal resultant, which gives the minimal value of $R_{\varphi^\gamma}$ among all $\PGL_2(K)$-conjugates of $\varphi$.
\end{notation} 

\section{Preliminary Lemmas} \label{sect:lemmas}
\subsection{The Resultant Under Iteration}
Our ultimate goal is to understand when the minimal resultant transforms ``nicely" under iteration. Therefore this basic lemma about how the resultant transforms under iteration is essential to what follows. It appeared in the first author's thesis (\cite{KJThesis}, Lemma 3.4), albeit with different notation. Its proof is straightforward, so we have included it here.

Let $\rho_D \in O(2D)$ be the resultant form. Let $$N=\frac{d^n(d^n-1)}{d(d-1)}.$$

\begin{lem}
\label{resultant}
If $(a,b)$ are the $2d+2$ coefficients of two homogeneous polynomials of degree $d$, and $(a_n,b_n)$ are the $2d^n+2$ coefficients of the two homogeneous polynomials of degree $d^n$ obtained by iteration $n$ times, then $\rho_{d^n} (a_n, b_n) = \rho_{d}(a,b)^N$. 
\end{lem}
\begin{proof}
This follows from an exercise in \cite{silverman:ads} that gives the resultant for a composition of pairs of homogeneous polynomials in two variables: let $f$, $g$ be of degree $n_1$ and $F$, $G$ of degree $n_2$. Then if $R = F(f,g)$ and $S = G(f, g)$, then $$\rho_{n_1n_2}(R, S) = \rho_{n_1}(f, g)^{n_2}\rho_{n_2}(F,G)^{n_1^2}$$
Now induct on $n$: the base case is trivial. Let $F,G$ be the homogenous polynomials corresponding to $(a,b)$ and let $F_n, G_n$ be the homogeneous polynomials of degree $d^n$ obtained by the $n$-th iteration. Then 
			\begin{align*}
			\rho_{d^{n+1}}(F_{n+1}, G_{n+1})&= \rho_{d}(F, G)^{d^n} \rho_{d^n}(F_n,G_n)^{d^2} \\
									   &= \rho_{d}(F,G)^{d^n + d^2(d^{n-1}(d^{n-1}+\dots + 1))} \\
									   &= \rho_{d}(F,G)^{d^n (d^{n+1}-1)/(d-1)}
			\end{align*}
\end{proof}

Now we move to understand when the resultant divisor of $\varphi$ transforms nicely under iteration, without any assumptions yet about minimality:
\begin{prop}
\label{minimaliterate}
The reduction $\varphi_v$ lies outside of $I(d)$ if and only if for every $n$ we have $R_{\varphi^n} = N\cdot R_{\varphi}$. 
\end{prop}
\begin{proof}



Let $n \in \NN$, and assume that $\varphi_v$ lies outside of $I(d)$. Let $(a,b) \in \AA^{2d+2}$ be normalized coefficients for $p$ with respect to $\varphi$. Let $(a_n, b_n) \in \AA^{2d^n+2}$ be the coefficients obtained by iteration. These coefficients are homogenous polynomials in the original coefficients; generically, $I(d)$ is the locus precisely where they all simultaneously vanish. Thus because $\varphi_v \notin I(d)$ (and because iteration commutes with taking reductions) at least one of the coefficients $(\widetilde{a_n}, \widetilde{b_n})$ evaluates to something non-zero. Therefore $(a_n, b_n)$ are normalized coefficients for $p$ with respect to $\varphi^n$, and $[\widetilde{a_n}, \widetilde{b_n}]=(\varphi^n)_v = \Gamma_n(\varphi_v)$, and so by Lemma \ref{resultant}, we get the desired formula.

Conversely, assume the formula holds for each $n$, and fix a particular $n$. 
Let $(A_n,B_n)$ be $2d^n+2$ \textit{indeterminates} corresponding to coefficients of a generic map of degree $d^n$; for $n=1$, write $(A_1, B_1) = (A,B)$. Let $(a,b)=(a_1,b_1)$ be the coefficients of a particular map $\varphi$ and let $(a_n,b_n)$ be the coefficients obtained by iteration. Also let $(A_n(A,B), B_n(A,B))$ be generic coefficients for an iterate map (i.e. plug in the iterate formula for each coefficient). Note $(A_n(a,b), B_n(a,b))=(a_n,b_n)$.  Let $c_n \in K$ be chosen so that $(c_na_n,c_nb_n)$ are normalized. Let $H_n(A_n,B_n)$ be a homogeneous form of degree $d_n$ such that $$
|H_n(c_na_n, c_nb_n))|_v = 1$$ Such a form must exist because $[\widetilde{c_n}\cdot\widetilde{a_n}, \widetilde{c_n}\cdot\widetilde{b_n}]$ is by assumption a well defined point of projective space; in fact, this form can simply be linear, taken to be one of the coefficients that, by assumption, doesn't vanish. The reason for leaving it in a more general form is that later we will reference this proof in a slightly different context (see Proposition \ref{ssiterates} below). Now, factoring out $c_n$, this says 
\begin{equation}|c_n|_v^{d_n}|H_n(a_n,b_n)|_v = 1 \label{unit}
\end{equation}

To show that $\varphi_v \notin I(d)$ is to show that $|c_n| = 1$: This is the same reasoning as above: if the coefficients are already normalized, their reduction after iteration has a non-zero entry, and so the reduction before iteration lies outside of $I(d)$). So we are done, by  (\ref{unit}), if we can show that $H_n(a_n,b_n)$ is non-vanishing at $p$. Let 
\[
\sigma = \frac{H_n(A_n,B_n)^{2d^n}}{\rho_{d^n}(A_n, B_n)^{d_n}}
\]
The exponents are chosen so that $\sigma$ has total degree zero. Evaluating $\sigma$ on $(c_na_n,c_nb_n)$ gives: 

\[
\sigma(c_na_n, c_nb_n) = \frac{H_n(c_na_n, c_nb_n)^{2d^n}}{\rho_{d^n}(c_na_n,c_nb_n)^{d_n}}
\] 
On the other hand, we have:
\begin{align*}
\sigma &= \frac{H_n(A_n(A,B), B_n(A,B))^{2d^n}}{\rho_{d^n}(A_n(A,B), B_n(A,B))^{d_n}} \\
	&= \frac{H_n(A_n(A,B), B_n(A,B))^{2d^n}}{\rho_d(A,B)^{Nd_n}} \textrm{     (By Lemma \ref{resultant}}) \\
\end{align*}
Then since $\sigma$ is total degree $0$,
\begin{align*}
	\sigma(c_na_n, c_nb_n) &= \sigma(a_n, b_n) \\&= \sigma(A_n(a,b), B_n(a,b)) \\ &= \frac{H_n(A_n(a,b), B_n(a,b))^{2d^n}}{\rho_d(a,b)^{Nd_n}} \\
	& = \frac{H_n(a_n, b_n)^{2d^n}}{\rho_d(a,b)^{Nd_n}}
\end{align*}

Thus:

\[
\frac{H_n(a_n, b_n)^{2d^n}}{\rho_d(a,b)^{Nd_n}} = \frac{H_n(c_na_n, c_nb_n)^{2d^n}}{\rho_{d^n}(c_na_n,c_nb_n)^{d_n}}
\]

Taking valuations, the right hand side is equal to $0-N\cdot R_{\varphi} \cdot d_n$ by the choice of $H_n(A_n,B_n)$ and the assumption about what the resultant is equal to (note that the coefficients are normalized). The left hand side, on the other hand, involves normalized coefficients for $\varphi$, and so its valuation is $\ord(H_n(a_n,b_n)^{2d^n})- N\cdot R_{\varphi} \cdot d_n$. Hence $|H_n(a_n,b_n)| = 1$ and we are done.

%
\end{proof}


\subsection{Semi-stability} \label{semistability}

To address the question of minimality, we will invoke a connection between semistability and minimality of the resultant.

 In \cite{silverman:space}, Silverman studied the GIT quotient $\Md$ of $\Ratd$ by the conjugation action of $\SL_2$. Crucial to this construction is the semi-stable locus $(\PP^{2d+1})^{ss}$, an open subscheme of $\PP^{2d+1}$ that contains $\Ratd$. Intuitively, it is the largest subscheme of $\PP^{2d+1}$ on which a quotient scheme makes sense. The following is a useful explicit way to think of the semi-stable locus. Let $A_d = \ZZ[a,b]$. $\Md$ and $\Ratd$ are affine schemes, defined over $\ZZ$, and the map between them is given by the map of rings $(A_d)_{(\rho_d)}^{\SL_2} \rightarrow (A_d)_{(\rho_d)}$ (the superscript indicates $\SL_2$ invariant functions). The quotient space $\Md^{ss}$ is $\mathrm{Proj}(A_d^{\SL_2})$ and $(\PP^{2d+1})^{ss}$ is simply the largest open set of $\PP^{2d+1}$ on which the inclusion of graded rings $A_d^{\SL_2} \rightarrow A_d$ induces a morphism of schemes. In this way, the semi-stable points are the complement of the indeterminacy locus for the quotient map.
 
 In \cite{STW}, the second author with Szpiro and Tepper proved that maps on $\mathbb{P}^1$ with semi-stable reduction have minimal valuation for the resultant:
 
 \begin{prop} \label{semistableimpliesminimal}
 Suppose that $\varphi$ has semi-stable reduction at $p$. Then $R_{\varphi} =R_{[\varphi]}$.
 \end{prop}
 \begin{proof}
 See \cite{STW}, Theorem 3.3. 
 \end{proof}
 In \cite{Rumely}, Rumely proves the same result using very different techniques.
 
A concrete description of the semi-stable points is provided by Silverman in \cite{silverman:space}, using the Hilbert-Mumford numerical criterion. For now, let $k$ be any field and let $F, G$ be the homogenous polynomials corresponding to $[a,b] \in \PP^{2d+1}(k)$. Write $A=$gcd$(F,G)$, and let $F=A \cdot F_0, G= A\cdot G_0$, and write $\widetilde{\varphi}$ for the morphism $[F_0, G_0]$ of $\PP^1(k)$. Following DeMarco in \cite{DeMarco2}, the point $[a,b]\in\PP^{2d+1}(k)$ has a \textit{hole} $h\in \PP^1(k)$ of depth $r$ if $h$ is a root of $A$ of multiplicity $r$. The numerical criterion essentially says that semi-stable points have holes of limited depth:
 
 \begin{prop} \label{numericalcriterion}
 If $d$ is even then $[a,b] \in \PP^{2d+1}(k)$ is in the semistable locus iff $d$ has a hole of depth $\leq d/2$, and if the depth of the hole $h$ is $d/2$, then  then there is the additional requirement that  $\widetilde{\varphi}(h) \neq h$.  If $d$ is odd then $[a,b] \in \PP^{2d+1}(k)$ is in the semistable locus iff $d$ has a hole of depth $\leq (d+1)/2$, and if the depth of the hole $h$ is $(d+1)/2$, then there is the additional requirement that $\widetilde{\varphi}(h) \neq h$. 
 \end{prop}
 \begin{proof}
 See \cite{silverman:space}, Proposition 2.2; we've formulated the statement following DeMarco \cite{DeMarco2}. 
 \end{proof}

Note that the semi-stable locus contains some constant maps, and thus may overlap with $I(d)$. In fact, DeMarco shows that the semi-stable locus and $I(d)$ have non-trivial intersection for all $d > 2$. Also, clearly the complement of $I(d)$ contains many points that are not semi-stable. Further, semi-stability is not always preserved under iteration. 

In the next section, we will describe a probability measure, introduced by DeMarco, which is connected with the numerical criterion above. The measure gives the asymptotic behavior of the depth of the hole at $z$ for higher iterates, relative to the degree of the iterate. Using this measure, the numerical criterion implies that maps stay semi-stable under iteration if and only if there are no points with mass greater than $\frac{1}{2}$ (see Propositions 3.2 and 3.3 of \cite{DeMarco2}).

\section{Proof of the Main Theorem}\label{mainresult}

In this section, we prove the local version of the main theorem:

\begin{thm}\label{thm:mainthm}
Let $K$ be a complete, algebraically closed non-Archimedean valued field, and let $\varphi\in K(z)$ have degree $d\geq 2$. 

The minimal resultant iteration formula $R_{[\varphi^n]} = N \cdot R_{[\varphi]}$ holds if and only if in any coordinate system where $\varphi$ has semistable reduction, we have that $\varphi_v \not\in I(d)$ and $\varphi^n$ has semistable reduction for all $n$. 
\end{thm}
\begin{proof}
Suppose first that 
\begin{equation}\label{eq:minresiterformula}
R_{[\varphi^n]} = N\cdot R_{[\varphi]}\ ,
\end{equation} and fix coordinates so that $\varphi$ has semistable reduction. Let $\varphi = [F,G]$ be a normalized lift of $\varphi$, with $$F(X,Y) = a_d X^d + ... + a_0 Y^d\ , \ G(X,Y) = b_d X^d + ... + b_0 Y^d\ .$$ By Proposition~\ref{semistableimpliesminimal}, we find
\[
R_{[\varphi]} = \ord(\Res(F,G))\ .
\]
If we pass to an iterate $\varphi^n$, we may encounter two problems: first, $\varphi^n$ might no longer attain semistable reduction; second, the corresponding lift $\varphi^n$ may not be normalized. Nevertheless, there is a formula for computing the normalized resultant (see, e.g., \cite{Ru1} Equation (8)):
\begin{align*}
R_{\varphi^n} &= \ordRes(F^n, G^n) - 2d^n \min_{0\leq i, j\leq d}(\ord(a_i^n), \ord(b_j^n))\ ,
\end{align*} where $a_i^n, b_j^n$ are the coefficients of the coordinate polynomials of $\varphi^n = [F^n, G^n]$. Using the iteration formula for the resultant given in Lemma~\ref{resultant} above gives
\begin{align*}
R_{\varphi^n} & = N \ordRes(F,G) - 2d^n \min_{0\leq i, j \leq d^n}(\ord(a_i^n), \ord(b_j^n))\ .
\end{align*}
 Now, suppose $\varphi^n$ does not have semistable reduction. Then $R_{[\varphi^n]} <  R_{\varphi^n}$, and we find
\begin{align*}
N\cdot \ord\Res(F,G) &= N\cdot R_{[\varphi]})\\
& = R_{[\varphi^n]} \\
 &< R_{\varphi^n}\\
 & = N \ord\Res(F,G) - 2d^n \min_{0\leq i, j \leq d^n} \min(\ord(a_i^n), \ord(b_j^n))\ .
\end{align*} Cancelling the common factor of $N\cdot \ord\Res(F,G)$ and reversing the inequality gives
\begin{equation}\label{eq:contradiction}
0 > 2d^n \min_{0 \leq i, j \leq d^n} \min(\ord(a_i^n), \ord(b_j^n))\ ;
\end{equation} but recall that our lift $\varphi = [F,G]$ of $\varphi$ was normalized, and the coefficients $a_i^n, b_j^n$ are polynomial combinations of the coefficients of $F,G$. Taking polynomial combinations cannot decrease the ord value, hence (\ref{eq:contradiction}) is a contradiction. We conclude that $\varphi^n$ has semistable reduction as well.

In particular, (\ref{eq:minresiterformula}) now reads
\[ 
R_{\varphi^n} = R_{[\varphi^n]} = N\cdot R_{[\varphi]} = N\cdot R_{\varphi}\ , 
\] and so by Proposition~\ref{minimaliterate} we conclude that $\varphi_v \not\in I(d)$. This completes the proof of the forward implication of the main theorem.

For the reverse implication, suppose that we have chosen a coordinate system in which $\varphi^n$ has semistable reduction for all $n$, and also for which $\varphi_v\not\in I(d)$. Combining Propositions~\ref{minimaliterate} and~\ref{semistableimpliesminimal} gives
\[
R_{[\varphi^n]} = R_{\varphi^n} = N\cdot R_{\varphi} = N\cdot R_{[\varphi]}\ ,
\] which is the asserted equality.
\end{proof}

\section{An Equivalent Condition: Barycenters and Minimal Resultant Locus} \label{Berk}
In this section we give two geometric conditions on the Berkovich line $\pberk$ which are equivalent to the conditions given in the main theorem:

\begin{cor} \label{cor:Berkcorollary}The following are equivalent:
\begin{enumerate}
\item[(A)] The minimal resultant iteration formula $R_{[\varphi^n]} = N\cdot R_{[\varphi]}$ holds for every $n$.
\item[(B)] For each $n$, $\MinResLoc(\varphi) \subseteq \MinResLoc(\varphi^n)$, and for every $\zeta = \gamma(\zetaG)$ in $\MinResLoc(\varphi)$, we have $(\varphi^\gamma)_v \not\in I(d)$. 
\item[(C)] There exists $\zeta\in \Bary(\mu_\varphi)$ with $\zeta = \gamma(\zetaG)$ and $(\varphi^\gamma)_v \not\in I(d)$.
\end{enumerate} 
\end{cor}

In the following section we will define the sets $\MinResLoc(\cdot)$ and $\Bary(\mu_\varphi)$, as well as the points $\zeta, \zetaG$ in the Berkovich projective line $\pberk$.

\subsection{The Berkovich Projective Line}
The Berkovich projective line over $K$ is formally defined to be the collection \footnote{Technically one considers equivalence classes of seminorms, where two seminorms are equivalent if there is a constant $C>0$ with $[G]_1 = C^d[G]_2$ for all homogeneous polynomials $G\in K[X,Y]$ of degree $d$; see \cite{BakerRumely} Section 2.2.} of multiplicative seminorms $[\cdot]$ on $K[X,Y]$ which extends the absolute value on $K$ and which are non-vanishing on the ideal $(X,Y)$. 

Berkovich has shown (see \cite{Ber}) that each point $\zeta\in \pberk$ is one of four types:
\begin{itemize}
\item Points of Type I correspond to \textit{ evaluation seminorms} given by $[F(X,Y)]_\zeta = |F(a, b)|$ for any point $(a,b)\in \PP^1(K)$. This gives a natural inclusion of $\PP^1(K)$ into $\pberk$.
\item Points of Type II or Type III correspond to \textit{ disc seminorms} given $$[F(X,Y)]_\zeta = \sup_{w\in D(a,r)} |F(w, 1)|\ ;$$ here, $D(a,r)$ is a closed subdisc in $K$. If $r\in |K^\times|$ then $\zeta$ is said to be of Type II, while if $r\not\in |K^\times|$ the point $\zeta$ is said to be of Type III. Such points are denoted $\zeta_{a,r}$. 
\item Points of Type IV can be obtained as limits of disc seminorms: $$[F(X,Y)]_\zeta = \lim_{i\to\infty} [F(X,Y)]_{\zeta_{a_i, r_i}}\ ,$$ where $...D(a_i, r_i)\supseteq D(a_{i+1}, r_{i+1})...$ is a decreasing family of discs for which $\cap_i D(a_i, r_i) = \emptyset$, but $\lim_{i\to\infty} r_i >0$. 
\end{itemize}

As a special case, the point $\zeta_{0,1}$ corresponding to the unit disc is called the Gauss point, and is often denoted $\zetaG$. This name arises from the fact that $\sup_{w\in D(0,1)} |F(w,1)|$ is equal to the maximum of the coefficients of $F$, which is the Gauss norm of $F$. 

\subsubsection{The action of a rational map on $\pberk$}
The action of a rational map $\varphi\in K(z)$ on $\PP^1(K)$ extends naturally to give an action on $\pberk$; as an important case, the automorphism group $\PGL_2(K)$ of $\PP^1(K)$ extends to act on $\pberk$ (\cite{BakerRumely} Corollary 2.13) If $\zeta= \zeta_{a,r}$ is a type II point, then it satisfies $\gamma(\zetaG) =\zeta_{a,r}$, where $\gamma(z) = bz+a$ for any $b\in K^\times$ with $|b| = r$. 

\subsubsection{Tangent Spaces on $\pberk$}
A tangent vector at a point $\zeta\in \pberk$ is an equivalence class of paths emmanating from $\zeta$, where two paths $[\zeta, P], [\zeta, Q]$ are equivalent if they share a common initial segment. The collection of all tangent vectors at $\zeta$ is denoted $T_\zeta$. 

The tangent directions at $\zeta\in \pberk$ are in one-to-one correspondence with the connected components of $\pberk \setminus \{\zeta\}$, and as such we often denote these components by $B_\zeta(\vv)^-$; in \cite{DF}, when $\zeta$ is of type II these sets are called Berkovich discs. When $\zeta$ is of type II, the tangent directions also correspond to elements of $\PP^1(k)$

The map $\varphi$ induces a tangent map $\varphi_*$ on $T_\zeta$, which has the following interpretation at $\zetaG$: assume first that the residue map $\widetilde{\varphi}$ is non-constant; identifying $\tilde{z}\in \PP^1(\overline{k})$ with tangent vectors $\vv_{\tilde{z}}\in T_{\zetaG}$, the tangent map $\varphi_*$ is given by $$\varphi_* \vv_{\tilde{z}} = \vv_{\widetilde{\varphi}(\tilde{z})}\ .$$ The general case -- when $\varphi$ does not have constant residue map, or at type II points other than $\zetaG$ -- can be obtained by pre- or post- composition of $\varphi$.

\subsubsection{Multiplicities}
There are two notions of multiplicity that play an important role in dynamics on $\pberk$; these are the directional multiplicity $m_\varphi(\zeta, \vv)$ and the surplus multiplicity $s_\varphi(\zeta, \vv)$.

For a fixed tangent direction $\vv\in T_{\zetaG}$, the image of the Berkovich disc $B_{\zetaG}(\vv)^-$ under $\varphi$ is either another Berkovich disc $B_{\varphi(\zetaG)}(\varphi_* \vv)^-$, or else is all of $\pberk$ (\cite{BakerRumely} Proposition 9.40). The surplus and directional multiplicities arise from this phenomenon when trying to count preimages: for any $y\in \PP^1(K)$, we have (see \cite{RL} Lemma 2.1, \cite{BakerRumely} Proposition 9.41, or \cite{Faber} Theorem 3.10) $$\#(\varphi^{-1}(y) \cap B_{\zetaG}(\vv_{\tilde{z}})^-) = \left\{\begin{matrix} m_\varphi(\zetaG, \vv_{\tilde{z}}) + s_\varphi(\zetaG, \vv_{\tilde{z}}), & y \in B_{\varphi(\zetaG)}(\varphi_* \vv_z)^-\\ s_\varphi(\zetaG, \vv_{\tilde{z}}), & y\not\in B_{\varphi(\zetaG)}(\varphi_* \vv_z)^-\end{matrix}\right.\ .$$

\subsection{Measures on $\pberk$}

In this section, we recall the construction of the equilibrium measure $\mu_\varphi$ on $\pberk$ attached to a rational map $\varphi\in K(z)$, and use it to induce a measure $\widetilde{\mu_\varphi}$ on $\PP^1(k)$ (endowed with the discrete topology). We then show how a construction of DeMarco can be used to give another measure $\mu_{\varphi_v}$ on $\PP^1(k)$, and we show that $\widetilde{\mu_\varphi} = \mu_{\varphi_v}$. 

The equilibrium measure $\mu_\varphi$ on $\pberk$ was defined (more or less) simultaneously by several different authors (\cite{FRLErgodic}, \cite{BRSmallHt}, \cite{CL}); it can be realized as the weak limit of the normalized pullback $\frac{1}{d^n} (\varphi^n)^* \delta_y$ of a point mass which charges a non-exceptional point $y\in \PP^1(K)$; here, the preimages are weighted according to their multiplicities. The equilibrium measure is the unique $\varphi$-invariant probability measure on $\pberk$ which satisfies $\frac{1}{d}\varphi^* \mu_\varphi = \mu_\varphi$ and which does not charge the exceptional set of $\varphi$. 

The support of the equilibrium measure is called the Julia set of $\varphi$; it is also characterized as the smallest totally invariant closed subset of $\pberk$ disjoint from the exceptional set of $\varphi$. Loosely speaking, it is the locus of points where iterates of $\varphi$ behave chaotically. We will make use of the Julia set in Section~\ref{examples} below.

Following DeMarco and Faber \cite{DF}, we can give an interpretation of the measure $\mu_\varphi$ in terms of the multiplicities introduced above. Let $U= B_{\zeta}(\vv)^-$ be a Berkovich disc, and let $g= \mathds{1}_U$ be the characteristic function of this set. Then for any point $y\in \PP^1(K)$, let $\nu_n =\frac{1}{d^n} (\varphi^n)^*\delta_y$; by the definition of the surplus and directional multiplicities, we find that $$\nu_n (U) = \frac{\epsilon(n, y,U) m_{\varphi^n}(\zeta, \vv) + s_{\varphi^n}(\zeta, \vv)}{d^n}\ ,$$ where $\epsilon(n, y,U)$ is either 0 or 1. By the weak convergence of the $\nu_n$, we have \begin{equation}\label{eq:measuremultiplicities}\mu_\varphi(U) = \lim_{n\to\infty} \nu_n(U) = \lim_{n\to\infty} \frac{\epsilon(n, y,U) m_{\varphi^n}(\zeta, \vv) + s_{\varphi^n}(\zeta, \vv)}{d^n}\ .\end{equation}

We are now ready to define the residue measure\footnote{The residue measure here is a special case of a $\Gamma$-measure introduced by DeMarco and Faber \cite{DF}, in the case that $\Gamma=\{\zetaG\}$.} $\widetilde{\mu_\varphi}$ on $\PP^1(k)$: endow $\PP^1(k)$ with the discrete topology, and let $$\widetilde{\mu_\varphi}(\{z\}) := \mu_\varphi(B_{\zetaG}(\vv_z)^-)\ ,$$ where we recall that $B_{\zetaG}(\vv_z)^-$ is the connected component of $\pberk\setminus \{\zetaG\}$ corresponding to the point $z\in \PP^1(k)\simeq T_{\zetaG}$. If $\varphi$ has good reduction, this measure is the zero measure; otherwise, it is a probability measure on $\PP^1(k)$.

We now recall another probability measure on $\PP^1(k)$ induced by $\varphi_v$, which was first introduced by DeMarco \cite{DeMarco1} in the context of degenerating families of rational maps. Her work was carried out over $\CC$ but the construction holds more generally over any complete algebraically closed field.

As above, let $[a, b]\in \PP^{2d+1}(K)$ denote a normalized set of  coefficients of $\varphi$, and let $\varphi_v$ denote the rational map corresponding to the coefficientwise reduction $[\tilde{a}, \tilde{b}]\in \PP^{2d+1}(\overline{k})$. As a map, we can think of $\varphi_v$ as $\varphi_v = \tilde{A}\cdot \widetilde{\varphi}$. Following DeMarco \cite{DeMarco1}, if the degree of $\widetilde{\varphi}$ is nonzero, let $$\mu_{\varphi_v} := \sum_{n=0}^\infty \frac{1}{d^{n+1}} \sum_{\substack{\tilde{A}(h) = 0 \\ \widetilde{\varphi}^n(w) = h}} \delta_w\ ,$$ which is a probability measure supported at the iterated preimages (under $\widetilde{\varphi}$) of the holes of $\varphi_v$. If the degree of $\widetilde{\varphi}$ is zero, let $\mu_{\varphi_v}$ be the probability measure given $$\mu_{\varphi_v} = \frac{1}{d} \sum_{\tilde{A}(h) = 0} s_\varphi(\zetaG, \vv_h)\cdot  \delta_h\ .$$ The measure $\mu_{\varphi_v}$ can be viewed as the `equilibrium measure' for the degenerate rational map $\widetilde{\varphi}$. DeMarco shows
\begin{lem}[DeMarco, Corollary 2.3, \cite{DeMarco1}]
If $\varphi_v\not\in I(d)$, then $$\mu_{\varphi_v}(\{z\}) = \lim_{n\to\infty} \frac{s_{\varphi^n}(\zeta, \vv_{\tilde{z}})}{d^n}\ .$$
\end{lem}

The next proposition captures the relationship between $\widetilde{\mu_\varphi}$ and $\mu_{\varphi_v}$; note that in the case $K=\CC((t))$ is the field of Puiseux series, this result was proved in \cite{DF} Theorem B using different methods:

\begin{prop}\label{prop:measuresagree}
Let $\varphi\in K(z)$, and suppose that $\varphi$ does not have good reduction. Let $\widetilde{\mu_\varphi}, \mu_{\varphi_v}$ be the measures on $\PP^1(k)$ defined above. Then $\widetilde{\mu_\varphi} = \mu_{\varphi_v}$ as measures on $\PP^1(k)$. 
\end{prop}
\begin{proof}
We will work with two cases. First, assume that $\varphi_v \not\in I(d)$. Then by the previous lemma and  (\ref{eq:measuremultiplicities}) above, we find 
\begin{align*}
\widetilde{\mu_\varphi}(\{z\}) & = \mu_\varphi(B_{\zetaG}(\vv_z)^-)\\
& = \lim_{n\to\infty} \frac{\epsilon(n,y,U) m_{\varphi^n}(\zetaG, \vv_z) + s_{\varphi^n}(\zetaG, \vv_z)}{d^n}\\
& \geq \lim_{n\to\infty} \frac{s_{\varphi^n}(\zetaG, \vv_z)}{d^n} = \mu_{\varphi_v}(\{z\})\ .
\end{align*} But both $\widetilde{\mu_\varphi}, \mu_{\varphi_v}$ are probability measures on $\PP^1(\overline{k})$; since $\widetilde{\mu_\varphi} \geq \mu_{\varphi_v}$, they must be equal.

Now assume that $\varphi_v\in I(d)$; in particular, $\varphi$ has constant residue map $\tilde{\varphi} \equiv \tilde{c}$ for some $\tilde{c}\in \PP^1(k)$. We claim that $\widetilde{\mu_\varphi} \geq \mu_{\varphi_v}$ in this case as well. Assume instead that $\mu_{\varphi_v}> \widetilde{\mu_\varphi}$; then we can find a point $\tilde{z}_1\in \PP^1(k)$ with
\begin{align*}
\mu_{\varphi_v}(\{\tilde{z}_1\})& > \widetilde{\mu_\varphi}(\{\tilde{z}_1\})\ .
\end{align*} We now invoke Lemma 4.8 of \cite{DF}: let $\Gamma = \{\zetaG\}$ and $\Gamma' = \{\zetaG, \varphi(\zetaG)\}$, so that  $U=B_{\zetaG}(\{z_1\})^-$ is a $\Gamma$-disc. The measure $\mu_\varphi$, viewed as a $\Gamma'$ measure, satisfies the pullback relation in Lemma 4.8 of \cite{DF} (see Lemma 4.4, ibid), and we conclude that $$\widetilde{\mu_\varphi}(\{\tilde{z}_1\}) = \mu_\varphi(B_{\zetaG}(\vv_{z_1})^-) \geq \frac{s_{\varphi}(\zetaG, \vv_{z_1})}{d}\ .$$ But since $\tilde{\varphi}$ is constant, we have that $\mu_{\varphi_v}(\{\tilde{z}_1\}) = \frac{s_\varphi(\zetaG, \vv_{z_1})}{d} $. All together, we have $$\mu_{\varphi_v}(\{\tilde{z}_1\}) > \widetilde{\mu_\varphi}(\{\tilde{z}_1\}) \geq \frac{s_\varphi(\zetaG, \vv_{z_1})}{d} = \mu_{\varphi_v}(\{\tilde{z}_1\})\ ,$$ whichi is a contradiction. Therefore, $\widetilde{\mu_\varphi} \geq \mu_{\varphi_v}$; since both are probability measures, they must be equal.
\end{proof}

\subsubsection{Barycenters of Measures}
Let $\nu$ be a Borel probability measure on $\pberk$. The \textit{ barycenter} of $\nu$ is defined to be $$\Bary(\nu) = \{\zeta\in \pberk \ | \ \nu(B_\zeta(\vv)^-) \leq \frac{1}{2}\ \textrm{ for all }\vv\in T_\zeta\}\ .$$ This set is always non-empty, and is either a point or a segment in $\hberk$ (this fact is originally due to Rivera-Letelier; see \cite{KJThesis} Proposition 4.4 for a proof). We are most interested in the case that $\nu = \mu_\varphi$.

\subsection{The Minimal Resultant Locus}
The last tool we need in proving Corollary~\ref{cor:Berkcorollary} is the minimal resultant locus cooresponding to the map $\varphi$, which is denoted by $\MinResLoc(\varphi)$. This object was first introduced by Rumely (see \cite{Ru1}, \cite{Rumely}) in his study of minimal models of rational maps.

Let $[F,G]$ be a normalized homogeneous lift of $\varphi$, and define $$\ordRes(\varphi) = -\log_v |\Res(F,G)|\ .$$ This quantity $\SL_2(\mathcal{O}_K)$-invariant, in the sense that $\ordRes(\varphi) = \ordRes(\varphi^\gamma)$ for any $\gamma\in \SL_2(\mathcal{O}_K)$. Since $\SL_2(\mathcal{O}_K)$ is the stabilizer of $\zetaG$ in $\SL_2(K)$ (\cite{Ru1} Proposition 1.1), one obtains a well-defined function on type II points by $$\ordRes_\varphi(\zeta) = \ordRes(\varphi^\gamma)\ ,$$ where $\zeta = \gamma(\zetaG)$. Rumely shows that it extends to a well-defined, continuous, convex-up function on all of $\pberk$ (\cite{Ru1} Thoerem 0.1). The set of points where this function attains its minimum is denoted $\MinResLoc(\varphi)$. By Proposition \ref{semistableimpliesminimal}, points $\zeta$ corresponding to conjugates realizing semistable reduction all lie in $\MinResLoc(\varphi)$; Rumely extended this to show that the converse is also true:

\begin{prop}[Rumely, \cite{Rumely}, Theorem 7.4]\label{prop:minresss}
Suppose $\varphi(z)\in K(z)$ has degree $d\geq 2$. Let $P\in \hberk$ be a point of type II, and let $\gamma\in \GL_2(K)$ be such that $P = \gamma(\zetaG)$. Then
\begin{itemize}
\item $($A$)$ $P$ belongs to $\MinResLoc(\varphi)$ if and only if $\varphi^\gamma$ has semistable reduction.
\item $($B$)$ If $P$ belongs to $\MinResLoc(\varphi)$, then $P$ is the unique point in $\MinResLoc(\varphi)$ if and only if $\varphi^\gamma$ has stable reduction. 
\end{itemize}
\end{prop}
 \noindent This result is complementary to Proposition~\ref{semistableimpliesminimal} above.

We will also make use of a result from the first author's thesis concerning the asymptotic behaviour of the sets $\MinResLoc(\varphi^n)$. The metric $\sigma$ given here is the big metric on the Berkovich hyperbolic line $\hberk$; see \cite{BakerRumely} Chapter 2 (they denote the metric $\rho$; we've chosen $\sigma$ here, as $\rho$ conflicts with our notation for the resultant form). 

\begin{prop}\label{prop:convergenceofMinResLoc}[\cite{KJThesis}, Proposition 4.1]
Let $\varphi\in K(z)$ be a rational map of degree $d\geq 2$. For any $\epsilon >0$, there exists an $\hat{N}$ such that for every $n\geq \hat{N}$ we have $$\MinResLoc(\varphi^n) \subseteq \mathcal{B}_{\epsilon}(\Bary(\mu_\varphi))\ ,$$ where $\mathcal{B}_{\epsilon}(A) = \{ \zeta\in \hberk \ : \ \min_{x\in A}\  \sigma(x,\zeta) < \epsilon\}$. 
\end{prop}
The main application of this result to the present work is that, if $\zeta\in \MinResLoc(\varphi^n)$ for all $n$, then $\zeta\in \Bary(\mu_\varphi)$.

\subsection{Proof of Corollary~\ref{cor:Berkcorollary}}
\begin{proof}
Suppose first that the minimal resultant iteration formula holds. Then by Theorem~\ref{thm:mainthm}, given a coordinate system in which $\varphi$ has semistable reduction we must also have that $\varphi^n$ has semistable reduction. Applying Proposition~\ref{prop:minresss}, this says that any $\zeta\in \MinResLoc(\varphi)$ must lie in $\MinResLoc(\varphi^n)$ for every $n$, hence the containment $\MinResLoc(\varphi)\subseteq \MinResLoc(\varphi^n)$ for all $n$. The assertion that $(\varphi)_v\not\in I(d)$ for such coordinates follows from Theorem~\ref{thm:mainthm}. This concludes the proof of (A)$\implies$(B).

Now suppose that for all $n$, $\MinResLoc(\varphi)\subseteq \MinResLoc(\varphi^n)$, and that for any $\zeta\in \MinResLoc(\varphi)$ with $\zeta = \gamma(\zetaG)$, we have $(\varphi^\gamma)_v\not\in I(d)$. By Proposition~\ref{prop:convergenceofMinResLoc} any $\epsilon$-ball of $\Bary(\mu_\varphi)$ must contain all $\MinResLoc(\varphi^m)$ for $m$ larger than or equal to some threshold $N$. Thus, each point $\zeta\in\MinResLoc(\varphi)$ lies in $\Bary(\mu_\varphi)$, and our hypothesis in this case guarantees that $(\varphi^\gamma)_v\not\in I(d)$ for $\zeta=\gamma(\zetaG)$. This completes the proof of (B)$\implies$(C).

Finally, suppose that there is some point $\zeta\in \Bary(\mu_\varphi)$ with $(\varphi^\gamma)_v\not\in I(d)$ and $\zeta = \gamma(\zetaG)$. Without loss of generality we may assume that $\zeta = \zetaG$. The condition $\zetaG\in \Bary(\mu_\varphi)$ gives that $\widetilde{\mu_\varphi}(\{z\}) \leq \frac{1}{2}$ for all points $z\in \PP^1(k)$. By Proposition~\ref{prop:measuresagree}, this implies that $\mu_{\varphi_v}(\{z\}) \leq \frac{1}{2}$ for $z\in \PP^1(k)$. In particular, since $(\varphi)_v \not\in I(d)$, Propositions~3.2 and~3.3 of \cite{DeMarco2} imply that $(\varphi^n)_v$ is semistable for all $n$. Thus, by Theorem~\ref{thm:mainthm}, we conclude that the minimal resultant iteration formula holds. This completes (C)$\implies$(A).

\end{proof}

\section{An Equivalent Condition in Algebraic Geometry}\label{sect:AG}

The conditions given in our main theorem are equivalent to two other sets of conditions, which we give here as corollaries. The first equivalent set of conditions we will give is really just a rephrasing of what we already have in more geometric language, which may yield useful insight. There is a natural diagram of graded rings:

%

\[
\xymatrix{
A_d^{\SL_2} \ar[r]  & A_d  \\
A_{d^n}^{\SL_2} \ar[u] \ar[r] &  A_{d^n} \ar[u]
}
\]
%

The vertical maps are given by the iteration morphism, which preserves $\SL_2$ invariance because iteration commutes with the group action. If we apply $\mathrm{Proj}$ to the entire diagram then we get, passing from top right to bottom left, a morphism that is defined on an open set $U_n$ of $\PP^{2d+1}$: 
\[
U_n \rightarrow (\mathcal{M}_{d^n})^{ss}
\]

If we now base change to $k$, to get a diagram of varieties, then $U_n$ consists of all maps that lie outside of $I(d)$, are semi-stable, and for which the $n$-th iterate is semi-stable. If we then take the intersection $\cap_n U_n$, we get a set for which all iterates are semi-stable.

\begin{cor}
\label{Unforall}
A map $\varphi$ has minimal resultant satisfying $(R_{[\varphi^n]})_v=N \cdot (R_{[\varphi]})_v$ if and only if there exists of choice of coordinates for which the reduction $\varphi_v$ is in $U_n$, the complement of the indeterminacy locus of the rational map $\PP^{2d+1} \rightarrow (\mathcal{M}_{d^n})^{ss}$, for every $n$.
\end{cor}

\begin{proof}
This is immediate from the preceding comments and Theorem \ref{mainthm}. 
\end{proof}

Being an intersection of infinitely many Zariski open sets, the set $\cap_n U_n$ is somewhat mysterious; a priori we can't say much about it's topology. An interesting question is what the codimension of its complement is. Knowing this would give us a better sense of ``how many" maps satisfy the minimal resultant formula. 

Exploiting the $\SL_2$ invariance in the above diagram gives another result which relates the semi-stability of an iterate to that of the original map:
\begin{prop}
\label{ssiterates}
Suppose that for some $n$, $\varphi^n$ has semi-stable reduction at $p$, and that the resultant iteration formula $$R_{\varphi^n} = N\cdot R_{\varphi}$$ holds for this $n$. Then:
	\begin{enumerate}
	\item $\varphi$ itself has semi-stable reduction, and thus the minimal resultant iteration formula $R_{[\varphi^n]}= N\cdot R_{[\varphi]}$ holds for this particular $n$. 
	\item The resultant iteration formula $R_{\varphi^n} = N\cdot R_{\varphi}$ holds for every other $n$ as well. 
	\end{enumerate}
\end{prop}
\begin{proof}

We are in the situation of the proof of the backwards direction of Proposition \ref{minimaliterate}; let $(a,b)$ be normalized coefficients as in this proof, and follow the proof through, noting that because $\varphi^n$ has semi-stable reduction, the non-vanishing form $H_n(A_n,B_n)$ can be chosen to be $\SL_2$ invariant (the existence of such a form is one characterization of the semi-stable locus). Then the image of this form in $A_d^{\SL_2}$ under the map in the above diagram gives an $\SL_2$ invariant form for which $\varphi$ doesn't vanish, by the same calculation in the proof of Proposition \ref{minimaliterate}. Thus $\varphi$ has semi-stable reduction. This gives the first statement. Noting that this calculation gives, as before, that $\varphi_v \notin I(d)$, we see that by the forward direction of Proposition $\ref{minimaliterate}$ the second statement follows. 
\end{proof}


\section{An Application to Potential Theory}\label{sect:potentialtheory}

As another corollary of Theorem~\ref{mainresult}, we are able to obtain a formula for the minimal value of the diagonal Arakelov-Green's function $g_\varphi(x,x)$, provided the equivalent conditions of Theorem~\ref{thm:mainthm} hold. 

Given a probability measure $\nu$ on $\pberk$, the (normalized) Arakelov-Green's function attached to $\nu$ is $$g_\nu(x,y) = \int_{\pberk} -\log\delta(x,y)_\zeta d\nu(\zeta) + C\ ;$$ here, $\delta(x,y)_\zeta$ is the Hsia kernel which measures the distance between $x$ and $y$ relative to the basepoint $\zeta$ (see \cite{BakerRumely} Chapter 4 for the definition of the Hsia kernel, and \cite{BakerRumely} Chapter 8 for a discussion of the Arakelov-Green's function on $\pberk$).The constant $C$ is chosen so that $$\iint g_\nu(x,y) d\nu(x)d\nu(y) = 0\ .$$ 

In the case that $\nu = \mu_\varphi$ is the equilibrium measure associated to $\varphi$, we simply write $g_\varphi (x,y) = g_{\mu_\varphi}(x,y)$. Under the hypotheses of our theorem, we obtain:

\begin{cor}
Let $K$ be a complete, algebraically closed non-Archimedean valued field, and let $\varphi\in K(z)$ have degree $d\geq 2$. If $\varphi$ satisfies either of the equivalent conditions in Theorem~\ref{thm:mainthm}, then
\begin{equation}\label{eq:minAGformula}
\min_{x\in \pberk} g_\varphi(x,x) = \frac{1}{d(d-1)}R_{[\varphi]}\ .
\end{equation}

In this case, $\min_{x\in \pberk} g_\varphi(x,x) >0$ if and only if $\varphi$ does not have good reduction. 
\end{cor}

\noindent\textit{Remark:} In regards to the last statement of the corollary, Matt Baker has obtained a more general result that holds unconditionally; see \cite{Baker} Theorem 3.15. 

\begin{proof}
If $\varphi$ satisfies the equivalent conditions of Theorem~\ref{thm:mainthm}, then after a suitable change of coordinates $\varphi^n$ has semistable reduction for each $n$, and in particular 
\begin{align*}
\min_{x\in \pberk}\ \frac{1}{d^{n}(d^n-1)} \ordRes_{\varphi^n}(x) &= \frac{1}{d^n(d^n-1)}R_{[\varphi^n]}\\
 & =\frac{1}{d(d-1)} R_{[\varphi]}\ .
\end{align*} 
\noindent By \cite{KJThesis} Corollary 4.8, we find that 
\begin{align*}
\min_{x\in \pberk} g_\varphi(x,x) &= \lim_{n\to\infty} \left(\min_{x\in \pberk} \frac{1}{d^{2n}-d^n} \ordRes_\varphi(x)\right)\\
& = R_{[\varphi]}\ .
\end{align*}

 The last assertion of the corollary follows immediately, since $\varphi$ has good reduction if and only if $R_{[\varphi]} = 0$.
\end{proof}

\section{Flexible Latt\`es Maps} \label{examples}
We illustrate our main theorem for the case of Latt\`es maps on Tate curves. Let $q\in K$ with $0<|q|<1$. The Tate curve corresponding to $q$ is an elliptic curve $E$ whose $K$ points are isomorphic (as a group) to $K^\times/q^\ZZ$. Tate \cite{Tate} gave an explicit parameterization of $E$: there exist meromorphic functions
\begin{align}
x(w) & = \frac{w}{(1-w)^2} + \sum_{m=1}^\infty \left(\frac{q^m w}{(1-q^m w)^2} + \frac{q^m w^{-1}}{(1-q^m w^{-1})^2} - 2 \frac{q^m}{(1-q^m)^2}\right)\ ,\\
y(w) &= \frac{w^2}{(1-w)^3} + \sum_{m=1}^\infty \left(\frac{q^{2m} w^2}{(1-q^m w)^3} - \frac{q^m w^{-1}}{(1-q^m w^{-1})^3} + \frac{q^m}{(1-q^m)^2}\right)\ .
\end{align} on $K$ which satisfy the relation 
\begin{equation}\label{eq:Tatecurve}
E: y^2 + xy = x^3 -b_2 x - b_3\ ,
\end{equation} for explicit series $b_2 = 5q + 45 q^2 + ...$ and $b_3 = q + 23q^2 + ...$. The Latt\`es map for multiplication-by-$m$ on $E$ is the unique map $\psi_m$ which completes the diagram

\[
\xymatrix{
E \ar[d]^x \ar[r]^{[m]}  & E\ar[d]^x  \\
\PP^1(K) \ar[r]^{\psi_m} & \PP^1(K)
}\ ,
\] where $x:E\to \PP^1(K)$ is projection onto the $x$ coordinate. Fitting this into a larger diagram, we find

\[
\xymatrix{
\aberk\setminus\{0\} \ar[d]^{x(w)} & K^\times \ar@{^{(}->}[l]\ar[d]^{x(w)} \ar[r]^{w\mapsto w^m} & K^\times \ar[d]^{x(w)}\\
\pberk  &\PP^1(K) \ar@{^{(}->}[l]\ar[r]^{\psi_m} & \PP^1(K)
}
\] where we are including $K^\times$ into its Berkovich analytification $\aberk \setminus\{0\}$. The extension of $x(w)$ to $\aberk\setminus\{0\}$ can be described as follows: let $[\cdot]_\zeta$ be a multiplicative seminorm on the algebra $K\langle T, T^{-1}\rangle$ of power series which converge on $K^\times$. If $\zeta\in \aberk\setminus\{0\}$ is a pole of $x$, then $x(\zeta) = \infty\in \pberk$. Otherwise, define $x(\zeta)$ by $$[F]_{x(\zeta)} := [F\circ x]_{\zeta} \ , \ \forall F\in K[T]\ .$$

\begin{lem}\label{lem:fixedpoints}
The segment $(0, \infty)$ in $\aberk$ is mapped under $x(w)$ to the segment $[\zetaG, \zeta_{0, |q|^{1/2}}]$. Consequently the Julia set for $\psi_m$ in $\pberk$ is the segment $[\zetaG, \zeta_{0, |q|^{1/2}}]$. There are $m$ type II fixed points in the Julia set given explicitly by 
$$\zeta_i = \left\{\begin{matrix} 
\zeta_{0, |q|^{\frac{i}{2(m+1)}}}\ , &  i \textrm{ even}\\
\zeta_{0, |q|^{\frac{i-1}{2(m-1)}}}\ , & i \textrm{ odd}\end{matrix}\right.
$$ for $i=1,2 ..., m$. 
\end{lem}

\noindent\textit{Remark:} The conclusion that the Julia set is a segment lying in $[0, \infty]\subseteq \pberk$ has already been established in \cite{FRLErgodic} Proposition 5.1; there, Favre and Rivera-Letelier show the existence of coordinates for which the Julia set is $[\zetaG, \zeta_{0, |q|^{-1/2}}]$ and explicitly describe the action of $\varphi$ on the Julia set; but they do not identify which coordinate system gives this Julia set, which we need later in this section. The novelty in the result here is that we are able to identify explicit coordinates on $\pberk$ for which the Julia set is given by $[\zetaG, \zeta_{0, |q|^{1/2}}]$.

\begin{proof}
In order to compute the image of the segment $(0,\infty)\subseteq \aberk$, we first note that $x(w)$ satisfies $x(qw) = x(w) = x(w^{-1})$; consequently, we can restrict our attention to $w\in K^\times$ satisfying $|q| < |w| < 1$. By determining $|x(w)|$ for generic $w$ in this region, we will be able to determine explicitly the image of $(0, \infty)$. 

Consider 
\begin{align}
x_1(w) &= \frac{w}{(1-w)^2}\ ,\\
x_{2, m}(w) & = \frac{q^m w}{(1-q^m w)^2}\ ,\\
x_{3, m}(w) & = \frac{q^m w^{-1}}{(1-q^m w^{-1})^2}\ ,\\
x_{4,m}& = 2\frac{q^m}{(1-q^m)^2}\ .
\end{align} Then $x(w) = x_1(w) + \sum_{m=1}^\infty (x_{2,m}(w) + x_{3,m}(w) - x_{4,m})$. 

Note that, since $|w|<1$, we have $|1-w| = 1$, and consequently $|x_1(w)| = |w|$. Next, we also find that $|1-q^m w| = 1$, and hence $|x_{2,m}(w)| = |q|^m |w| < |w| = |x_1(w)|$ for all $m\geq1$. 

As we are also assuming that $|q|<|w|$, we find that $|1-q^m w^{-1}| = 1$ as well, so that $|x_{3,m}(w)| = |q|^m |w|^{-1}$. Finally, note that $|x_{4,m}| = |q|^m$. In all:
\begin{align*}
|x_1(w)| &= |w| > |q|^m |w| =  |x_{2,m}(w)|\ ,\\
|x_{3,m}(w)| &= |q|^m |w|^{-1} > |q|^m = |x_{4,m}|\ .
\end{align*} And so it suffices to compare $|x_1(w)|$ and $|x_{3,m}(w)|$ for $|q| < |w| < 1$. Note that $|x_{3,m}(w)|$ is largest for $m=1$, so in fact we only need to complare $|x_1(w)| = |w|$ and $|x_{3,1}(w)| = |q|\cdot |w|^{-1}$. 

When $|q|^{1/2} < |w| < 1$, we find that $|x_1(w)| = |w| > |q|\cdot |w|^{-1} = |x_{3,1}(w)|$, and hence $|x(w)| = |w|$, while for $|q| < |w| < |q|^{1/2}$ we find that $|x_{3,1}(w)| = |q|\cdot |w|^{-1} > |w| = |x_1(w)|$, and hence $|x(w)| = |q|\cdot |w|^{-1}$. 

Geometrically, this says that the segment $(\zetaG, \zeta_{0, |q|^{1/2}})\subseteq \aberk$ is mapped by $x(w)$ to the segment $(\zetaG, \zeta_{0, |q|^{1/2}})\subseteq \pberk$, while the segment $(\zeta_{0, |q|^{1/2}}, \zeta_{0, |q|})\subseteq \aberk$ is mapped onto the same segment $(\zetaG, \zeta_{0, |q|^{1/2}})$ in the reverse orientation.

The fact that the segment $[\zetaG, \zeta_{0, |q|^{1/2}}]$ is the Julia set then follows from the fact that $(0, \infty)\subseteq \aberk$ is totally invariant under the map $z\mapsto z^m$; consequently the segment $[\zetaG, \zeta_{0, |q|^{1/2}}]$ is totally invariant under the induced Latt\`es map $\psi_m$. 

The action of the Latt\`es map on $\mathcal{J}$ was first descibed in \cite{FRLErgodic} Proposition 5.1; it is an $m$-fold tent map. Partition the interval $\mathcal{J} = [\zetaG, \zeta_{0, |q|^{1/2}}]$ into $m$ intervals $\mathcal{I}_i\ , \ i= 1, 2, ..., m$ of equal width. Each $\mathcal{I}_i$ maps bijectively onto $\mathcal{J}$. Parameterize the interval $\mathcal{J}$ by $t$, where $t$ is taken with respect to the logarithmic path distance, so that $t\in [0, (1/2)\log |q|]$. We find that, on each interval $\mathcal{I}_i$, the Latt\`es map has the form 
\[
t\mapsto \left\{\begin{matrix} -m (t-\frac{i}{2m}\log |q|)\ , & i \textrm{ even}\\
m (t-\frac{i-1}{2m} \log|q|)\ , & i \textrm{ odd}\end{matrix}\right.\ .
\]
Solving for the fixed points on each interval gives the $\zeta_i$ asserted in the statement of the Lemma.

\end{proof}

\subsection{Crucial Measures for Latt\`es Maps}
With this characterization of the Julia set and the fixed points, we are in a position to apply Rumely's crucial measures in order to identify the unique conjugate of $\varphi_m$ which attains semistable reduction. 

An essential tool in developing this theory is the Laplacian on a metrized graph $\Gamma$. Berkovich space $\pberk$ has a natural path distance function $\sigma$ which is referred to as the `big metric' which is invariant under the action of $\SL_2$. A connected graph $\Gamma\subseteq \pberk $ is said to be finite if any vertex has finitely many branch points, and any two points lie at finite $\sigma$-distance from one another. 

The metric $\sigma$ induces a measure-valued Laplacian $\Delta_\Gamma$ on the space of functions which are `of bounded differential variation' on $\Gamma$; continuous piecewise affine functions -- such as $\ordRes_{\varphi}$ -- are examples of such functions. There is a natural extension of $\Delta$ to functions defined on all of $\pberk$ by taking inverse limits of the $\Delta_\Gamma$; see \cite{BakerRumely} Chapters 3 and 5. 

Let $\Gamma_{\textrm{FR}}$ be the graph in $\pberk$ spanned by the type I fixed points and the type II repelling fixed points of $\psi_m$, and let $\Gamma = \widehat{\Gamma_{\textrm{FR}}}$ be a truncation of this tree obtained by removing segments near type I fixed points\footnote{See \cite{Rumely} for the explicit truncation. Since the type I fixed points are at infinte distance from points in $\hberk$, we need to remove them and segments leading up to them in order to apply the Berkovich Laplacian.}. Then $\Gamma$ is a connected finite subgraph of $\pberk$. Rumely \cite{Rumely} has shown that  the Laplacian of $\ordRes_{\psi_m}$ on $\Gamma$ can be given $$\Delta_{\widehat{\Gamma_{FR}}} \ordRes_{\psi_m} = 2(d^2-d) (\mu_{\textrm{Br}, \Gamma} - \nu_m)\ , $$ Here, $\mu_{\textrm{Br},\Gamma}$ is the branching measure on the finite graph $\Gamma$ given by $$\mu_{\textrm{Br}, \Gamma} = \frac{1}{2}\sum_{P\in \Gamma} (2-v_\Gamma(P))\delta_P\ ,$$ where $v_\Gamma(P)$ is the valence of $P$ in $\Gamma$ and $\delta_P$ is the Dirac point-mass at $P$. The measure $\nu_m$ is the crucial measure associated to $\psi_m$; it is a discrete probability measure with finite support in $\Gamma = \widehat{\Gamma_{\textrm{FR}}}$, and as such can be written $$\nu_m = \frac{1}{d-1} \sum_{P\in \pberk} w_{\psi_m}(P) \delta_P$$ for explicit weight functions given in \cite{Rumely} Definition 8. 

\begin{prop}\label{prop:MRL}
The weight $w_{\psi_m}$ of each type II fixed point in the Julia set $[\zetaG, \zeta_{0, |q|^{1/2}}]$ can be given explicitly by
$$
w_\varphi(\zeta_i) = \left\{\begin{matrix}
m-1\ ,& i\textrm{ odd}\\
m+1\ , & i \textrm{ even and } i\neq m\\
m\ , & i=m\textrm{ is even}\end{matrix}\right.
$$
These are the only points in $\pberk$ with $w_{\psi_m}(\zeta)>0$. Consequently, the minimal resultant locus of $\psi_m$ consists of a single point and is given by
$$
\MinResLoc(\psi_m) = \left\{\begin{matrix}
\left\{\zeta_{0, |q|^{1/4}}\right\}\ , & m\textrm{ is odd}\\
\left\{\zeta_{0, |q|^{\frac{m}{4(m+1)}}}\right\}\ , &m\equiv 0 \mod 4\\
\left\{\zeta_{0, |q|^{\frac{m+2}{4(m+1)}}}\right\}\ ,&m\equiv 2 \mod 4\ .
\end{matrix}\right.
$$
\end{prop}
\begin{proof}
The weights of the fixed points $\zeta_i$ were computed by the first author in \cite{KJThesis} Example 2; note that the indices here are shifted by 1 from the ones there. There, it was also shown that these are the only points which receive weight. 

To determine the minimal resultant locus from the weights, we rely on the fact that the minimal resultant locus is the barycenter of the crucial measures (\cite{Rumely} Theorem 7.1). In our context, this says that

 \begin{center}
  ``$\zeta$ is in the minimal resultant locus if and only if the sum of the weights of $\zeta_i$ lying closer to $\zetaG$ than $\zeta$ is at most $\frac{m^2-1}{2}$, and the sum of the weights of $\zeta_i$ lying farther from $\zetaG$ than $\zeta$ is also at most $\frac{m^2-1}{2}$.''
\end{center} 

First suppose that $m$ is odd; we claim that $\zeta_{(m+1)/2} = \zeta_{0, |q|^{1/4}}$ is the unique point of the minimal resultant locus. There are $(m-1)/2$ weighted points which lie nearer to $\zetaG$ than $\zeta_{(m+1)/2}$, while there are $(m-1)/2$ weighted points which lie farther from $\zetaG$ than $\zeta_{(m+1)/2}$.
\begin{itemize}
\item[Case i:] If $m\equiv 1 \mod 4$, then exactly half of the indices $i = 1, ..., (m-1)/2$ are odd while exactly half are even. Therefore, the total contribution of weight from points $\zeta_i$ lying nearer to $\zetaG$ than $\zeta_{(m+1)/2}$ is  $$\frac{m-1}{4} \cdot (m-1) + \frac{m-1}{4} \cdot (m+1) = \frac{m(m-1)}{2} < \frac{m^2-1}{2}\ .$$ Similarly, the total weight of points lying farther from $\zetaG$ is also $\frac{m(m-1)}{2} < \frac{m^2-1}{2}$. Consequently, the point $\zeta_{(m+1)/2}$ is in the barycenter of the crucial measure, and hence is in the minimal resultant locus. 

If either of these sums is increased by $w_{\psi_m}(\zeta_{(m+1)/2}) = m-1$, it will exceed the threshold $\frac{m^2-1}{2}$. Geometrically, this says that if we deviate from $\zeta_{(m+1)/2}$ in any direction, we will not be in the barycenter; hence $\zeta_{(m+1)/2}$ is the unique point in the barycenter.

\item[ Case ii:] If $m\equiv 3\mod 4$, the argument is essentially the same, however the counts are slightly different. Here, there are $(m+1)/4$ indices among $i=1, ..., (m+1)/2$ which are odd, while there are $(m-3)/4$ such indices which are even; in total, the mass of the points lying nearer to $\zetaG$ is $$\frac{m+1}{4} \cdot (m-1) + \frac{m-3}{4} \cdot (m+1) = \frac{m^2-m-2}{2}< \frac{m^2-1}{2} \ .$$ Similarly, the total mass of points lying farther from $\zetaG$ than $\zeta_{(m+1)/2}$ is $\frac{m^2-m-2}{2} < \frac{m^2-1}{2}$; thus $\zeta_{(m+1)/2}$ is in the minimal resultant locus. However, if either of these are increased by $w_{\psi_m}(\zeta_{(m+1)/2}) = m+1$, then the total weight exceeds the threshold $\frac{m^2-1}{2}$. Consequently, $\zeta_{(m+1)/2}$ is the unique point in the minimal resultant locus. 
\end{itemize}

We now consider the case that $m$ is even, which we again partition into two cases depending on $m\mod 4$:

\begin{itemize}
\item[Case iii:] If $m\equiv 0 \mod 4$, and consider the point $\zeta_{m/2} = \zeta_{0, |q|^{\frac{m-1}{4(m+1)}}}$. Among the indices $i=1, 2, ..., (m/2) -1$, there are $m/4$ odd indices, which correspond to points with weight $m-1$, and there are $(m/4) - 1$ even indices which correspond to points with weight $m+1$. Thus, in total, among the weighted points lying nearer to $\zetaG$ than $\zeta_{m/2}$ we find a total mass of $$(m/4)(m-1) + ((m/4)-1)(m+1) = \frac{m^2-2m-2}{2} < \frac{m^2-1}{2}\ .$$ Among the indices $i=(m/2) + 1 ,..., m$, we find that there are $(m/4)$ odd indices and $(m/4)$ even indices; keeping in mind that when $i=m$ the weight is $w_{\psi_m}(\zeta_m) = m$, we find that the total weight among the $\zeta_i$ with $i=(m/2) + 1, ..., m$ is $$(m/4)(m-1) + ((m/4)-1)(m+1) + m = \frac{m^2-2}{2} < \frac{m^2-1}{2}\ .$$ Consequently, $\zeta_{m/2}$ lies in the barycenter of the crucial measures, and hence in the minimal resultant locus. It is the unique point, since if either sums of weights given above is incremented by $w_{\psi_m}(\zeta_{m/2}) = m+1$, then the total weight is greater than $(m^2-1)/2$. 

\item [Case iv:] Finally, if $m\equiv 2\mod 4$, we consider the point $\zeta_{(m/2)+1} = \zeta_{0, |q|^{\frac{m+2}{4(m+1)}}}$. Note that among the indices $i=1, 2, ..., (m/2)$, exactly $(m+2)/4$ are odd while $(m-2)/4$ are even. Consequently, the total weight among the corresponding $\zeta_i$ is $$\frac{m+2}{4}(m-1) + \frac{m-2}{4} (m+1) = \frac{m^2-2}{2} < \frac{m^2-1}{2}\ .$$ Among the indices $i=(m/2)+2, ..., m$, exactly $(m-2)/4$ are even, while exactly $(m-2)/4$ are odd. Again keeping in mind that the last interval has weight $m$, we find that the total weight among these $\zeta_i$ is $$\frac{m-2}{4} (m-1) + \left(\frac{m-2}{4} - 1\right) (m+1) + m = \frac{m^2-2m-2}{2} < \frac{m^2-1}{2}\ .$$ Consequently, $\zeta_{(m/2)+1}$ is in the minimal resultant locus, and arguing as in the previous cases, it is the unique such point. 
\end{itemize}

\end{proof}

Our next task is to determine the value of $\ordRes_{\psi_m}$ at the unique point in $\MinResLoc(\psi_m)$. A priori, this requires that we first write an explicit formula for $\psi_m$, then conjugate by the appropriate map, normalize the coefficients, and compute the resultant. There is a more combinatorial approach using Arakelov-Green's functions and Rumely's crucial measures.

Recall that the Arakelov-Green's function of a probability measure $\nu$ can be expressed
\[
g_\nu(x,y) = \int -\log \delta(x,y)_\zeta d\nu(\zeta) + C
\] for an appropriately chosen normalization constant. The Hsia kernel satisfies several change-of-variables formulas (see \cite{BakerRumely} Chapter 4); we recall two of these now. For a fixed $\zeta\in \hberk$,
\begin{align}
\delta(x,y)_\zeta &= \frac{\delta(x,y)_{\zetaG}}{\delta(x,\zeta)_{\zetaG}\cdot\delta(y, \zeta)_{\zetaG}}\label{eq:HsiatoGauss}\\
\delta(x,y)_{\zetaG} &= \frac{\delta(x,y)_{\infty}}{\delta(x,\zetaG)_\infty \cdot \delta(y, \zetaG)_\infty}\ ,\ \forall x,y\neq \infty\ .\label{eq:Hsiatoinfty}
\end{align} These formulas allow us to write the Arakelov-Green's function in several different ways; for example, if $\nu$ is a measure which does not charge $\infty$, then
\begin{align}
g_\nu(x,y) &= \int_{\pberk} -\log \frac{\delta(x,y)_\infty}{\delta(x,\zeta)_\infty \cdot \delta(y, \zeta)_\infty} +C\nonumber \\
& = -\log  \delta(x,y)_\infty - u_\nu(x,\infty) - u_\nu(y, \infty) + C \label{eq:AGalt}\ ,
\end{align} where $u_\nu(x,\infty) = \int_{\pberk} -\log \delta(x,\xi)_\infty d\nu(\xi)$ is a potential function for the measure $\nu$; in particular, it satisfies $\Delta u_\nu(\cdot, \infty) = \nu - \delta_\infty$.

We define the un-normalized, diagonal Arakelov-Green's function attached to the crucial measures to be
\[
\hat{g}_m(x) := \frac{1}{m^2-1}\sum_{i=1}^m w_{\psi_m}(\zeta_i) \cdot (-\log \delta(x,x)_{\zeta_i})
\] which agrees with $g_{\nu_{\psi_m}}(x,x)$ up to the additive constant $C$. We record for later use that by applying (\ref{eq:HsiatoGauss}), this can be rewritten as
\begin{equation}\label{eq:ghatchanged}
\hat{g}_m(x) := -\log \delta(x,x)_{\zetaG} + \frac{2}{m^2 -1}\sum_{i=1}^m w_{\psi_m}(\zeta_i)\cdot\log \delta(x, \zeta_i)_{\zetaG}\ .
\end{equation}

\begin{lem}\label{lem:oRGreen}
Let $\Gamma= \widehat{\Gamma_{\textrm{FR}}}$ be as above. Then $$\Delta_\Gamma \hat{g}_m(x) =  2(\mu_{\textrm{Br}, \Gamma} - \nu_m)\ .$$ Consequently, for $\zeta\in \Gamma$ we have $$\ordRes_{\psi_m}(\zeta) = m^2(m^2-1)\hat{g}_m(x) + \ordRes_{\psi_m}(\zetaG)\ .$$
\end{lem}
\begin{proof}
Notice that $g_{\nu_m}(x,x) = \hat{g}_m(x) + C$; consequently it will be enough to compute $\Delta_\Gamma g_{\nu_m}(\cdot, \cdot)$. By (\ref{eq:AGalt}), we may rewrite $g_{\nu_m}(x,x)$ as $$g_{\nu_m}(x,x) = -\log \delta(x,x)_\infty - 2u_{\nu_m}(x,\infty) + C\ .$$ We now compute the Laplacian of each term appearing in this expression: by \cite{KJThesis} Lemma 3.12, we have 
\begin{equation}\label{eq:Lapdeltadiag}
\Delta_\Gamma (-\log \delta(x,x)_\infty) = 2\mu_{\textrm{Br}, \Gamma} - 2\delta_{r_\Gamma(\infty)}\ ,
\end{equation} where $r_{\Gamma}:\pberk \to \Gamma$ is the following retraction map: given a point $\zeta\in \Gamma$, the point $r_\Gamma(x)$ is the first point along the segment $[x, \zeta]$ lying in $\Gamma$. This is well defined independent of $\zeta$ since $\pberk$ is uniquely path connected and $\Gamma$ is connected.

Since $u_{\nu_m}(\cdot, \infty)$ is a potential for $\nu_m$, we find: 
\begin{equation}\label{eq:Lappotential}
\Delta_\Gamma u_{\nu_m}(\cdot, \infty) = \nu_m - \delta_{r_\Gamma(\infty)}
\end{equation} and (\ref{eq:Lappotential}) gives $$\Delta_\Gamma \hat{g}_m = \Delta_\Gamma g_{\nu_m}(\cdot, \cdot) = 2(\mu_{\textrm{Br}, \Gamma} - \nu_m)\ .$$

In order to establish the last assertion of the lemma, recall that Rumely has shown (\cite{Rumely} Corollary 6.5) that $$\Delta_\Gamma \ordRes_{\psi_m} = 2(d^2 - d) (\mu_{\textrm{Br}, \Gamma} - \nu_m)\ .$$ Consequently, since $d=m^2$ for Latt\`es maps, we find \begin{equation}\label{eq:ordresisghat}
\ordRes_{\psi_m}(x) =m^2(m^2-1) \hat{g}_m(x) + C
\end{equation} for all $x\in \Gamma = \widehat{\Gamma_{\textrm{FR}}}$. To determine the constant, we will explicitly compute both sides at $\zetaG$; note that $\zetaG\in \Gamma=\widehat{\Gamma_{\textrm{FR}}}$, since $w_{\psi_m}(\zetaG) = w_{\psi_m}(\zeta_1) = m-1$. 

To determine $\hat{g}_m(\zetaG)$ first note that the Hsia kernel relative to $\zetaG$ has the following geometric interpretation: given two points $x,y\in \pberk$, consider the paths $[x, \zetaG], [y, \zetaG]$ connecting $x$ and $y$ to $\zetaG$. There is a unique point $w$ which lies in both paths and which is farthest from $\zetaG$. Then $$\delta(x,y)_{\zetaG} = q_v ^{-\lambda(w, \zetaG)}\ , $$ where $\lambda$ is the logarithmic path distance on $\pberk$. Therefore, $\hat{g}_m(\zetaG)$ can be evaluated using the expression in (\ref{eq:ghatchanged}):
\begin{align*}
\hat{g}_m(\zetaG) &= -\log \delta(\zetaG, \zetaG)_{\zetaG} + \frac{2}{m^2-1} \sum_{i=1}^m w_{\psi_m}(\zeta_i) \log \delta(\zetaG, \zeta_i)_{\zetaG}\\ 
& = 0\ ,
\end{align*} where the last equality follows from the fact that $\delta(\zetaG, \cdot)_{\zetaG} \equiv 1$. Therefore, $C=\ordRes_{\psi_m}(\zetaG)$, which completes the proof. 
\end{proof}

Therefore, in order to compute $\ordRes_{\psi_m}$ at the unique point of the minimal resultant locus, it suffices to (i) compute the value of $\hat{g}_m(\zeta)$ at the same point, and (ii) determine the normalized resultant of $\psi_m$.

\begin{prop}\label{prop:minGreenVal}
Let $\zeta_*$ be the unique point in $\MinResLoc(\psi_m)$. Then $$\hat{g}_m(\zeta_*) =\frac{1}{8}\log|q| \cdot \left\{ \begin{matrix}1\ , & m \textrm{ odd}\\
\frac{m^3+m^2-2m}{(m+1)(m^2-1)}\ , & m\textrm{ even}\ .
\end{matrix}\right.
$$

\end{prop}
\begin{proof}
In order to prove this lemma, we will use the expression for $\hat{g}_m$ given in (\ref{eq:ghatchanged}):
$$
\hat{g}_m(x) = -\log \delta(x,x)_{\zetaG} + \frac{2}{m^2-1}\sum_{i=1}^m w_{\psi_m}(\zeta_i) \log \delta(x, \zeta_i)_{\zetaG}\ .
$$
Since $\zeta_*, \zeta_i$ both lie in the segment $[\zetaG, \zeta_{0, |q|^{1/2}}]$, the terms $\log \delta(x,\zeta_i)_{\zetaG}$ can be decomposed as
\begin{equation}
\log \delta(\zeta_*, \zeta_i)_{\zetaG} = \left\{\begin{matrix}
-\sigma(\zeta_i, \zetaG)\ , & i < *\\
-\sigma(\zeta_*, \zetaG)\ , & i \geq *\end{matrix}\right. \ .
\end{equation} Therefore, we can further decompose $\hat{g}_m(\zeta_*)$ as
\begin{equation}\label{eq:ghatlastdecomp}
\hat{g}_m(\zeta_*) = \sigma(\zeta_*, \zetaG)\left(1-\frac{2}{m^2-1}\sum_{i\geq *} w_{\psi_m}(\zeta_i)\right) - \frac{2}{m^2-1} \sum_{i < *} w_{\psi_m}(P) \sigma(\zeta_i, \zetaG)\ .
\end{equation}

We consider four cases depending on $m\mod 4$. Suppose first that $m$ is odd; then $\zeta_* = \zeta_{(m+1)/2}$ is the unique point in the minimal resultant locus. It satisfies $$-\log \delta(\zeta_*, \zeta_*)_{\zetaG} = \sigma(\zeta_*, \zetaG) = -\frac{1}{4} \log |q|\ .$$
\begin{itemize}
\item[Case i:] If $m\equiv 1\mod 4$, then $w_{\psi_m}(\zeta_*) =m-1 $. Arguing as in the proof of Proposition~\ref{prop:MRL}  we find that the sum of the weights of the $\zeta_i$ with $i \geq *$ was 
\begin{equation}\label{eq:m1geq}
\sum_{i \geq *} w_{\psi_m}(\zeta_i) = \frac{m(m-1)}{2} + m-1 = \frac{m^2+m-2}{2}\ .
\end{equation} We also saw in Proposition~\ref{prop:MRL} that among the indices with $i = 1, 2, ..., *-1$ exactly half are odd and half are even. Therefore, using the concrete expression for the points $\zeta_i$ given in Lemma~\ref{lem:fixedpoints}, we find
\begin{align}
\sum_{i < *} w_{\psi_m}(\zeta_i) \sigma(\zeta_i, \zetaG) &= \sum_{\substack{i < * \\ i \textrm{ odd}}} (m-1) \frac{i-1}{2(m-1)}\log|q|^{-1} + \sum_{\substack{ i < * \\ i \textrm{ even}}} (m+1) \frac{i}{2(m+1)}\log|q|^{-1}\nonumber\\
& = \frac{\log|q|^{-1}}{2}\sum_{i=1}^{(m-1)/2} i - \frac{\log|q|^{-1}}{2} \sum_{\substack{ i < *\\ i \textrm{ odd}}} 1\nonumber\\
& = \log|q|^{-1}\left(\frac{m^2-1}{16}  - \frac{m-1}{8}\right) = -\frac{m^2 - 2m+1}{16}\log|q|\ .\label{eq:m1leq}
\end{align} Inserting (\ref{eq:m1geq}) and (\ref{eq:m1leq}) into (\ref{eq:ghatlastdecomp}) gives
\begin{align*}
\hat{g}_m(\zeta_*) &= -\frac{1}{4} \log |q| \left(1-\frac{2}{m^2-1}\cdot \frac{m^2+m-2}{2}\right) +\frac{2}{m^2-1} \cdot \frac{m^2-2m+1}{16}\log|q|\\
&=-\frac{1}{4}\log|q| \left(\frac{1-m}{m^2-1}\right) + \frac{1}{8}\log|q| \cdot \frac{m^2-2m+1}{m^2-1}\\
& = \frac{1}{8}\log|q| \left(\frac{2m-2+m^2-2m+1}{m^2-1}\right) = \frac{1}{8} \log|q|\ .
\end{align*}
\item[Case ii:] If $m\equiv 3\mod 4$, then $\zeta_* = \zeta_{(m+1)/2}$ and $w_{\psi_m}(\zeta_*) = m+1$. Again we refer to Proposition~\ref{prop:MRL}, where we saw that the sum of the weights of the $\zeta_i$ with $i\geq *$ was 
\begin{equation}\label{eq:m3geq}
\sum_{i\geq *} w_{\psi_m}(\zeta_*) = \frac{m^2-2m-2}{2} + m+1 = \frac{m^2+m}{2}\ .
\end{equation} Among the indices with $i < *$, recall from the proof of Proposition~\ref{prop:MRL} that there are $(m+1)/4$ odd indices $i<*$; thus
\begin{align}
\sum_{i<*} w_{\psi_m}(\zeta_i) \sigma(\zeta_i, \zetaG) &= \sum_{\substack{i < * \\ i \textrm{ odd}}} (m-1) \frac{i-1}{2(m-1)}\log|q|^{-1} + \sum_{\substack{ i < * \\ i \textrm{ even}}} (m+1) \frac{i}{2(m+1)}\log|q|^{-1}\nonumber\\
& = \frac{\log|q|^{-1}}{2}\sum_{i=1}^{(m-1)/2} i - \frac{\log|q|^{-1}}{2} \sum_{\substack{ i < *\\ i \textrm{ odd}}} 1\nonumber\\
& = \log|q|^{-1}\left(\frac{m^2-1}{16}  - \frac{m+1}{8}\right) \nonumber\\
&= -\frac{m^2 - 2m-3}{16}\log|q|\label{eq:m3leq}\ .
\end{align}

We now insert (\ref{eq:m3geq}) and (\ref{eq:m3leq}) into (\ref{eq:ghatlastdecomp}) to find that 
\begin{align*}
\hat{g}_m(\zeta_*) &= -\frac{1}{4} \log |q| \left(1-\frac{2}{m^2-1}\cdot \frac{m^2+m}{2}\right) +\frac{2}{m^2-1} \cdot \frac{m^2-2m-3}{16}\log|q|\\
&=-\frac{1}{4}\log|q| \left(\frac{-1-m}{m^2-1}\right) + \frac{1}{8}\log|q| \cdot \frac{m^2-2m-3}{m^2-1}\\
& = \frac{1}{8}\log|q| \left(\frac{2m+2+m^2-2m-3}{m^2-1}\right) = \frac{1}{8} \log|q|\ .
\end{align*}
\end{itemize}

This completes the proof for $m$ odd. We now turn to the case that $m$ is even:
\begin{itemize}
\item[Case iii:] If $m\equiv 0\mod 4$, then $\zeta_* = \zeta_{m/2} = \zeta_{0, |q|^{\frac{m}{4(m+1)}}}$ and $w_{\psi_m}(\zeta_*) = m+1$. In Proposition~\ref{prop:MRL}, we already computed the total weight among the $\zeta_i$ with $i>*$; therefore
\begin{equation}\label{eq:m0geq}
\sum_{i\geq *}w_{\psi_m}(\zeta_i) = \frac{m^2-2}{2} + m+1 = \frac{m^2+2m}{2}\ .
\end{equation}

Next, we recall again from the proof of Proposition~\ref{prop:MRL} that there are $m/4$ odd indices $i<*$; therefore using the explicit formula for the $\zeta_i$ given in Lemma~\ref{lem:fixedpoints} we find
\begin{align}
\sum_{i<*} w_{\psi_m}(\zeta_i) \sigma(\zeta_i, \zetaG) &= \sum_{\substack{i < * \\ i \textrm{ odd}}} (m-1) \frac{i-1}{2(m-1)}\log|q|^{-1} + \sum_{\substack{ i < * \\ i \textrm{ even}}} (m+1) \frac{i}{2(m+1)}\log|q|^{-1}\nonumber\\
& = \frac{\log|q|^{-1}}{2}\sum_{i=1}^{(m/2)-1} i - \frac{\log|q|^{-1}}{2} \sum_{\substack{ i < *\\ i \textrm{ odd}}} 1\nonumber\\
& = \log|q|^{-1}\left(\frac{m(m-2)}{16}  - \frac{m}{8}\right) \nonumber\\
&= -\frac{m^2 -4m}{16}\log|q|\label{eq:m0leq}\ .
\end{align}
Inserting (\ref{eq:m0geq}) and (\ref{eq:m0leq}) into (\ref{eq:ghatlastdecomp}) gives
\begin{align*}
\hat{g}_m(\zeta_*) &= -\frac{m}{4(m+1)} \log |q| \left(1-\frac{2}{m^2-1}\cdot \frac{m^2+2m}{2}\right) +\frac{2}{m^2-1} \cdot \frac{m^2-4m}{16}\log|q|\\
&=-\frac{m}{4(m+1)}\log|q| \left(\frac{-1-2m}{m^2-1}\right) + \frac{1}{8}\log|q| \cdot \frac{m^2-4m}{m^2-1}\\
& = \frac{1}{8}\log|q| \left(\frac{4m^2+2m+(m^2-4m)(m+1)}{(m^2-1)(m+1)}\right) \\
&= \frac{m^3+m^2-2m}{m(m^2-1)}\cdot \frac{1}{8} \log|q|\ .
\end{align*}

\item[Case iv:] If $m\equiv 2\mod 4$, then $\zeta_* = \zeta_{(m/2)+1} = \zeta_{0, |q|^{\frac{m+2}{4(m+1)}}}$, and we find $w_{\psi_m}(\zeta_*) = m+1$. In Proposition~\ref{prop:MRL} we computed the total weight among $\zeta_i$ with $i>*$; therefore
\begin{equation}\label{eq:m2geq}
\sum_{i\geq *}w_{\psi_m}(\zeta_i) = \frac{m^2-2m-2}{2} + m+1 = \frac{m^2}{2}\ .
\end{equation}

We recall again from the proof of Proposition~\ref{prop:MRL} that there are $(m+2)/4$ indices $i< *$ which are odd; therefore
\begin{align}
\sum_{i<*} w_{\psi_m}(\zeta_i) \sigma(\zeta_i, \zetaG) &= \sum_{\substack{i < * \\ i \textrm{ odd}}} (m-1) \frac{i-1}{2(m-1)}\log|q|^{-1} + \sum_{\substack{ i < * \\ i \textrm{ even}}} (m+1) \frac{i}{2(m+1)}\log|q|^{-1}\nonumber\\
& = \frac{\log|q|^{-1}}{2}\sum_{i=1}^{m/2} i - \frac{\log|q|^{-1}}{2} \sum_{\substack{ i < *\\ i \textrm{ odd}}} 1\nonumber\\
& = \log|q|^{-1}\left(\frac{m(m+2)}{16}  - \frac{m+2}{8}\right) \nonumber\\
&= -\frac{m^2 -4}{16}\log|q|\label{eq:m2leq}\ .
\end{align}
Finally, inserting (\ref{eq:m2geq}) and (\ref{eq:m2leq}) into (\ref{eq:ghatlastdecomp}) we find
\begin{align*}
\hat{g}_m(\zeta_*) &= -\frac{m+2}{4(m+1)} \log |q| \left(1-\frac{2}{m^2-1}\cdot \frac{m^2}{2}\right) +\frac{2}{m^2-1} \cdot \frac{m^2-4}{16}\log|q|\\
&=-\frac{m+2}{4(m+1)}\log|q| \left(\frac{-1}{m^2-1}\right) + \frac{1}{8}\log|q| \cdot \frac{m^2-4}{m^2-1}\\
& = \frac{1}{8}\log|q| \left(\frac{2m+4+(m^2-4)(m+1)}{(m^2-1)(m+1)}\right) \\
&= \frac{m^3+m^2-2m}{m(m^2-1)}\cdot \frac{1}{8} \log|q|\ .
\end{align*}

\end{itemize}

\end{proof}

Finally, we are left to compute $\ordRes_{\psi_m}(\zetaG) = R_{\psi_m}$. To do this, we will conjugate $E$ into Weierstrass form and use known formulas for Latt\`es maps in Weierstrass form:
\begin{lem}\label{lem:Resval}
Let $\psi_m$ be the Latt\`es map associated to multiplication-by-$m$ on the elliptic curve E given in (\ref{eq:Tatecurve}). Then $$R_{\psi_m} = -\log |q|^{\frac{m^2(m^2-1)}{6}}\ .$$
\end{lem}
\begin{proof}
Recall that $E: y^2 + xy = x^3 - b_2 x - b_3$ for explicit power series $b_2 = 5q+ 45q^2 + ...$ and $b_3 = q + 23q^2 + ...$. The isogeny $\iota: (x,y) \mapsto \left(x-\frac{1}{12}, y-\frac{x}{2} + \frac{1}{24}\right)$ sends the affine plane curve $$\hat{E}: y^2 = x^3 -\left(\frac{1}{48} + b_2\right) x +\left( \frac{1}{864} +\frac{b_2}{12} - b_3\right)$$ to $E$. Let $g_2 = \frac{1}{48} + b_2$ and $g_3 = \frac{1}{864} + \frac{b_2}{12} - b_3$. Note that since we are assuming that $K$ does not have characteristic 2 or 3, we find $|g_2| = |g_3| = 1$. 

The map $\iota$ induces a commutative diagram

\[\xymatrixcolsep{5pc}
\xymatrix{
E\ar[r]^{\iota^{-1}}\ar[d]^{x}& \hat{E} \ar[r]^{[m]}\ar[d]^{x} & \hat{E} \ar[r]^{\iota}\ar[d]^{x} & E\ar[d]^{x}\\
\PP^1\ar[r]^{z\mapsto z+\frac{1}{12}}\ar@/_2pc/[rrr]_{\psi_m} & \PP^1 \ar[r]_{\phi_m} & \PP^1\ar[r]^{z\mapsto z-\frac{1}{12}} & \PP^1
}
\] where $\phi_m$ is the Latt\`es map induced by the x-coordinate map on the Weierstrass curve $\hat{E}$. Therefore, $\psi_m = \gamma^{-1}\circ \phi_m \circ \gamma$, where $\gamma(z) = z+\frac{1}{12}$. It is known (see \cite{silverman:ads} Exercise 6.23) that $\phi_m = \frac{f_m}{g_m}$ for polynomials $f_m, g_m \in \ZZ[x, g_2, g_3]$ where $f_m$ is monic of degree $m^2$ and $g_m$ has degree $m^2-1$. The natural homogeneous lift $\Phi_m = [ F_m, G_m]$ is therefore a normalized lift, and by \cite{silverman:ads} Exercise 6.23 we have that $$\Res(F_m, G_m) = \pm \Delta(\hat{E})^{\frac{m^2(m^2-1)}{6}}\ ,$$ where $\Delta(\hat{E})$ is the discriminant of $\hat{E}$.

Applying the conjugation in the above commutative diagram, we find, $$\psi_m = \frac{12f_m(x+1/12) - g_m(x+1/12)}{12 g_m(x+1/12)}\ .$$ We remark that the natural homogeneous lift $$\Psi_m = [12F_m(X+\frac{1}{12}Y, Y) - G_m(X+\frac{1}{12}Y), 12 G_m(X + \frac{1}{12}Y, Y)]$$ is again normalized, since $F_m$ is monic, $G_m$ has degree strictly smaller than $F_m$, and the characteristic of $K$ is not 2 or 3. Therefore, in order to compute $R_{\psi_m}$, note that $\SL_2$ conjugation leaves the resultant invariant, i.e. 
\begin{align*}
\Res\left(12F_m(X+\frac{1}{12}Y, Y) - G_m(X+\frac{1}{12}Y,Y), 12 G_m(X + \frac{1}{12}Y, Y)\right)= \Res(F_m, G_m)\ .
\end{align*} Finally, for Tate curves it is known that $|\Delta(\hat{E})| = |q|$ (see, e.g., \cite{Tate}); therefore $$R_{\psi_m} = -\log |q|^{\frac{m^2(m^2-1)}{6}}$$ as asserted.
\end{proof}

We are finally ready to give the expression for the minimal resultant value of a Latt\`es map:
\begin{prop}\label{prop:Latteseg}
Suppose that $K$ is a complete, algebraically closed, non-Archimedean valued field with characteristic and residue characteristic not equal to 2 or 3. Let $\psi_m$ be the Latt\`es map associated to a Tate curve $E$ with uniformizing parameter $q$ satisfying  $0<|q|<1$. Then
\begin{equation}\label{eq:minresLattes}
R_{[\psi_m]} = \left\{\begin{matrix} -\frac{m^2(m^2-1)}{24} \log |q|\ ,& m \textrm{ odd}\\
\left(\frac{m^5+m^4-2m^3}{8(m+1)} - \frac{m^2(m^2-1)}{6}\right)\log |q|\ , & m \textrm{ even}\ .
\end{matrix}\right. \ .
\end{equation}

In particular, the iteration formula $R_{[(\psi_m)^n]} = \frac{(m^{n})^2\left((m^n)^2-1\right)}{m^2(m^2-1)}R_{[\psi_m]}$ hold if and only if $m$ is odd.
\end{prop}
\begin{proof}
By Lemma~\ref{lem:oRGreen}, we can compute the minimal resultant value by
\[
R_{[\psi_m]} = \ordRes_{\psi_m}(\zeta_*) = m^2(m^2-1) g_m(\zeta_*) + \ordRes_{\psi_m}(\zetaG)\ .
\]
The quantity $g_m(\zeta_*)$ was computed in Proposition~\ref{prop:minGreenVal}, while the quantity $\ordRes_{\psi_m}(\zetaG) = R_{\psi_m}$ was computed in Lemma~\ref{lem:Resval}. Inserting these into the above expression for $R_{[\psi_m]}$ gives the asserted formula. The last claim is immediate from the given expression for $R_{[\psi_m]}$. 

\end{proof}

In particular, if we combine the above proposition with \cite{KJThesis} Corollary 4.8, we are able to give an explicit formula for the minimal of the diagonal Arakelov-Green's function attached to a Latt\`es map:
\begin{cor}
Let $K$ be as in Proposition~\ref{prop:Latteseg}, and let $\psi_m$ be the Latt\`es map associated to a Tate curve $E$ with uniformizing parameter $q$ satisfying  $0<|q|<1$. Then $$\min_{x\in \pberk} g_{\psi_m} (x,x) = -\frac{1}{24} \log |q|\ .$$
\end{cor}
\begin{proof}
By Proposition~\ref{prop:Latteseg}, the minimum resultant is given by the formula in (\ref{eq:minresLattes}). Passing to iterates, if the minimal resultant value of $\psi_m$ is normalized by $\frac{1}{d^{2n}-d^n} = \frac{1}{m^2(m^2-1)}$ then by \cite{KJThesis} Corollary 4.8, it converges to $\min_{x\in \pberk}g_\varphi(x,x)$. In particular, normalizing the expression in (\ref{eq:minresLattes}) and letting $n\to \infty$ we find  $$\min_{x\in \pberk} g_\varphi(x,x) = -\frac{1}{24} \log_v |q|\ .$$
\end{proof}

In the case of a number field, the minimal value of $g_{\psi_m}(x,x)$ given here can be compared with the minimal value of the Arakelov-Green's function on the elliptic curve $E$ itself. 

Let $L$ be a number field, and let $v$ be a finite place of $L$; denote the completion at $v$ by $L_v$. Let $K_v$ be the completion of the algebraic closure of $L_v$. The Arakelov-Green's function on an elliptic curve $E/K_v$ is given by $g_{E, v}(P,Q) = \lambda_v(P-Q)$, where $\lambda_v$ is the local N\'eron-Tate height on $E$. The idea of the next two Propositions and their proofs were suggested to us by Matt Baker:
\begin{prop}\label{prop:minGreenEC}
Let $E/K_v$ be a Tate curve with uniformizing parameter $q$, and let $g_{E,v}(P,Q) = \lambda_v(P-Q)$ be the normalized Arakelov-Green's function on $E$. Then $$\min_{P, Q\in E\times E}\ g_{E,v}(P,Q) = \frac{1}{24} \log |q|_v\ .$$
\end{prop}
\begin{proof}
We will use notation from Baker and Petsche \cite{BakerPetsche}. Let $u:E\to K_v^\times/q^\ZZ$ be the inverse of the Tate isomorphism described above, where we view $u(P)$ as an element of $K_v^\times$ normalized so that $|q|_v< |u(P)|_v \leq 1$. The retraction homomorphism $r:E\to \RR/\ZZ$ is given by $$r(P) = \frac{\log |u(P)|_v}{\log |q|_v}\ .$$ Since $E$ has multiplicative reduction, the Green's function $\lambda_v(P-Q)$ on $E$ can be computed explicitly in terms of the periodic second Bernoulli polynomial and a non-negative intersection term (\cite{BakerPetsche} Section 3):
\[
\lambda_v(P-Q) = i_v(P,Q) + \frac{1}{2} \textbf{B}_2(r(P-Q)) \cdot \log\max(|j_E|_v, 1)\ .
\] The periodic second Bernoulli polynomial $\textbf{B}_2(t) =(t-\lfloor t\rfloor)^2 -(t-\lfloor t\rfloor) + \frac{1}{6}$ is minimized for $t= \frac{1}{2}$, and the minimal value in this case is $\textbf{B}_2\left(\frac{1}{2}\right) = -\frac{1}{12}$. 

Now choose any $P, Q$ to satisfy $|u(P-Q)|_v = |q|_v^{1/2}$ and $r(P)\neq r(Q)$ (in particular, we could take $Q=O$ to be the identity and $P$ to be the image of $q^{1/2}$ under the Tate map $K_v^\times/q^\ZZ \to E$). The assumption $r(P)\neq r(Q)$ implies that the intersection term $i_v(P,Q) = 0$, and since $i_v(\cdot, \cdot)\geq 0$ our choice of $P$ and $Q$ gives a global minimum. We find that
\begin{align*}
\min_{P, Q\in E\times E} \ \lambda_v(P-Q) &= \frac{1}{2}\textbf{B}_2\left(\frac{1}{2}\right) \cdot \log\max(|j_E|_v, 1)\\
& = \frac{1}{24} \log |q|\ , 
\end{align*} where we are using the fact that $|j_E|_v = |q|_v^{-1}$ for a Tate curve.
\end{proof}

In the case of the Latt\`es map associated to multiplication by two, we can also compare our calculations for the minimum of $g_{\psi_2}(x,x)$ on $\pberk$ to the minimum of $g_{\psi_2}(x,y)$ on $\PP^1(K)$:
\begin{prop}
Let $E/K_v$ be the Tate curve over $K_v$ with uniformizing parameter $q,\  0 < |q| < 1$. Then $$\min_{x,y\in \PP^1(K_v)}\ g_{\psi_2}(x,y) = \frac{1}{12} \log_v |q|\ .$$
\end{prop}
\begin{proof}
As the minimum value of $g_{\psi_2}(x,y)$ is unchanged by conjugation, we can work instead with a minimal Weierstrass form $E':y^2 = x^3 + Ax + B$; write $\Omega:E' \to E$ for the (inverse of the) isogeny achieving the minimal Weierstrass form. Let $\varphi_2$ be the endomorphism on $\PP^1(K_v)$ induced by the multiplication-by-2 map on $E'$ via the projection onto the $x$ coordinate $x:E' \to \PP^1(K_v)$. 

In Appendix B of \cite{Baker}, Baker shows the following relationship between the local N\'eron-Tate of points on $E'$ and the Arakelov-Green's function attached to $\varphi_2$:
\[
g_{\varphi_2}(x(P), x(Q)) = \lambda_v (P-Q) + \lambda_v (P+Q)\ .
\]As in the proof of Proposition~\ref{prop:minGreenEC}, the right side of this expression is minimized if we choose $Q=O\in E'$ to be the identity and $P\in E'$ so that $|u(\Omega(P))| = |q|^{1/2}$ (this is the same $u$ as in the proof of Proposition~\ref{prop:minGreenEC}; as its domain is $E$, we need to first map $P$ from $E'$ to $E$).

We saw in the proof of Proposition~\ref{prop:minGreenEC}) that the value of $\lambda_v$ in this case is $$\lambda_v(P) = \frac{1}{24}\log_v |q|\ ;$$ since the projection $E'(K_v) \to \PP^1(K_v)$ is surjective, we find
\[
\min_{x,y\in \PP^1(K_v)}\ g_{\varphi_2}(x,y) = \frac{1}{12}\log_v |q|
\] which establishes the desired result.

\end{proof}

\subsection{Discussion}
The results of this section tell us that the Latt\`es maps $\psi_m$ provide examples both of where the conditions of our main Theorem are met and where they are not met. When $m$ is odd, the conjugate attaining semistable reduction is independent of the iterate $\psi_{m^n}$ (Proposition~\ref{prop:MRL}). As our main theorem predicts, the formula in (\ref{eq:minresLattes}) for the minimal resultant values transform according to the rule given in Theorem~\ref{thm:mainthm}. 

However, when $m$ is even, something different happens: the points in the minimal resultant locus of $\psi_{m^n}$ are different for each $n$, hence there is no one conjugate which has semistable reduction for all $n$ (Proposition~\ref{prop:MRL}). Consequently, the minimal resultant does not transform nicely in this case.

\subsection{Another Example}
We close with another example where the equivalent conditions of our theorem hold. Let $p\geq 3$ be a prime number and let $K=\CC_v$ be the $p$-adic complex numbers. Define $$\varphi(z) = \frac{z^p-z}{p}\ .$$ In \cite{RumelyCap}, it was shown that $\mu_\varphi$ is Haar measure on $\ZZ_v$; consequently the barycenter of $\mu_\varphi$ is precisely $\zetaG$. Moreover, a direct computation shows that the reduction $\varphi_v$ is not in $I(d)$. Therefore the equivalent conditions of Corollary~\ref{cor:Berkcorollary} hold; in particular, $\MinResLoc(\varphi)\subseteq \MinResLoc(\varphi^n)$ for all $n$. This can also be shown directly using the calculations in Example 1 of \cite{KJThesis}.

\end{document}